\documentclass{article}

\usepackage{lipsum}
\usepackage{amsfonts, amsmath, amsthm}
\usepackage{graphicx}
\usepackage{epstopdf}
\usepackage{algorithmic}
\usepackage{url}

\usepackage{amssymb}
\usepackage{mathrsfs}
\usepackage{tikz}
\usetikzlibrary[shapes.arrows, calc, positioning]
\usetikzlibrary{arrows}

\newcommand{\mob}{\textnormal{\texttt{mob}}}
\newcommand{\tot}{\textnormal{\texttt{tot}}}
\newcommand{\net}{\textnormal{\texttt{net}}}

\newcommand{\VV}{v}
\newcommand{\UU}{u}
\newlength{\tmpln}
\newlength{\tmplnn}
\newlength{\tmplnnn}

\ifpdf
 \DeclareGraphicsExtensions{.eps,.pdf,.png,.jpg}
\else
 \DeclareGraphicsExtensions{.eps}
\fi

\newtheorem{remark}{Remark}

\newtheorem{corol}{Corollary}
\newtheorem{theorem}{Theorem}
\newtheorem{lemma}{Lemma}
\newtheorem{prop}{Proposition}

\title{Singular perturbation for a two-class Processor-Sharing queue with impatience}

\author{R. Nasri\thanks{Orange Labs, OLN/GDM, Orange Gardens, 44 avenue de la R\'epublique, 
CS 50010, 92326 Chatillon Cedex (\texttt{ridha.nasri@orange.com}).}
\and F. Simatos\thanks{ISAE-SUPAERO, Université de Toulouse, 10 avenue Edouard Belin, 31055 Toulouse, France, (\texttt{florian.simatos@isae.fr}).}
\and A. Simonian\thanks{Orange Labs, OLN/GDM, Orange Gardens, 44 avenue de la R\'epublique, 
CS 50010, 92326 Chatillon Cedex (\texttt{alain.simonian@orange.com}).}}

\usepackage{amsopn}

\ifpdf
\fi

\begin{document}

\maketitle

% REQUIRED
\begin{abstract}
 % This is an example SIAM \LaTeX\ article. This can be used as a
 % template for new articles. Abstracts must be able to stand alone
 % and so cannot contain citations to the paper's references,
 % equations, etc. An abstract must consist of a single paragraph and
 % be concise. Because of online formatting, abstracts must appear as
 % plain as possible. Any equations should be inline.
A two-class Processor-Sharing queue with one impatient class is studied. Local exponential decay rates for its stationary distribution $(N(\infty), M(\infty))$ are established in the heavy traffic regime where the arrival rate of impatient customers grows proportionally to a large factor~$A$. This regime is characterized by two time-scales, so that no general Large Deviations result is applicable. In the framework of singular perturbation methods, we instead assume that an asymptotic expansion of the solution of associated Kolmogorov equations exists for large $A$ and derive it in the form
\[ \mathbb{P}(N(\infty) = Ax, M(\infty) = Ay) \sim \frac{g(x,y)}{2\pi A} \cdot e^{-A \, H(x,y)}, \qquad x > 0, \; y > 0, \]
with explicit functions $g$ and $H$. 

This result is then applied to the model of mobile networks proposed in~\cite{POAS16} and accounting for the spatial movement of users. We give further evidence of a unusual growth behavior in heavy traffic in that the stationary mean queue length $\mathbb{E}(N_\mob(\infty))$ and $\mathbb{E}M_\mob(\infty))$ of each customer-class increases proportionally to 
$$
\mathbb{E}(N_\mob(\infty)) \propto \mathbb{E}(M_\mob(\infty)) \propto \displaystyle \log\left(\frac{1}{1-\varrho_\tot}\right)
$$
with system load $\varrho_\tot$ tending to 1, instead of the usual $1/(1-\varrho_\tot)$ growth behavior.
%
% Unlike the classical Large Deviations framework, the considered setting is a case of Singular Perturbation for a bi-dimensional Markov birth-and-death process for which no general results are applicable. Denoting by $(N_A,M_A)$ the pair of queue occupancies for each customer class, a scaling on state variables and the assumed existence of an asymptotic expansion for large $A$ of the solution to the Kolmogorov equations enable us to successively derive
%
% $\bullet$ explicit expressions for the exponential decay rate
% $$
% H(x,y) = - \lim_{A \uparrow +\infty}
% \frac{1}{A} \, \log \mathbb{P}(N_A = Ax, M_A = Ay), \quad
% x \geqslant 0, y \geqslant 0,
% $$
% and the sharp asymptotics
% $$
% \mathbb{P}(N_A = Ax, M_A = Ay) \sim \frac{g(x,y)}{2\pi A} \cdot e^{-A \, H(x,y)},
% \quad x > 0, y > 0;
% $$
%
% $\bullet$ a weak convergence theorem for the scaled pair $(N_A,M_A)/A$, showing that its stationary distribution converges to a deterministic point $(x^*,y^*)$ and that the centered pair
% $(N_A - Ax^*,M_A-Ay^*)/\sqrt{A}$ is asymptotically Gaussian.
%
% These asymptotic results are applied to the closed-loop Processor-Sharing queue in high load condition, where impatient customers are fed back to the queue until the completion of their service. Remarkably, a logarithmic growth
% $$
% \mathbb{E}(N_\mob(\infty)) \sim - \frac{x^*}{H(0,0)} \log(1-\varrho_\tot), \quad
% \mathbb{E}(M_\mob(\infty)) \sim - \frac{y^*}{H(0,0)} \log(1-\varrho_\tot)
% $$
% of the closed-loop queue occupancy $(N_\mob(\infty),M_\mob(\infty))$ is exhibited when the total exogenous load $\varrho_\tot$ tends to 1.
\end{abstract}

% REQUIRED
% \begin{keywords}
% Singular perturbation, two time-scales Markov processes, Processor-Sharing queue with impatience.
% \end{keywords}

% REQUIRED
% \begin{AMS}
%  60F05, 60F10, 60J10, 60J20.
% \end{AMS}

%%%%%%%%%%%%%%%%%%%%%%%%%%%%%%
\section{Queuing Model and Main Results} \label{sec:model-main-results}
%%%%%%%%%%%%%%%%%%%%%%%%%%%%%
We describe the addressed queuing system and the specific asymptotic regime considered to evaluate its stationary occupancy distribution. We then state our main mathematical results and apply them to the account of spatial user movement in mobile networks.

%%%%%%%%%%%%%%%%%%%%%%%%%%%
\subsection{Two-class Processor-Sharing queue with one impatient class}
%%%%%%%%%%%%%%%%%%%%%%%%%%%
In this paper, we consider a two-class Markovian Processor-Sharing (PS) queue where one class of users are impatient and leave the system at rate $\theta > 0$. This queuing system is depicted in Figure~\ref{PSqueue} and can be described as follows:
\begin{itemize}
	\item the arrival process of patient (resp.\ impatient) customers entering the queue is Poisson with rate $\alpha$ (resp.\ $\beta$);
	\item service requirements for successive patient (resp.\ impatient) customers are i.i.d.\ and exponentially distributed with mean $1/\mu$ (resp.\ $1/\nu$);
	\item server capacity is normalized to unity and customers are served according to the PS service discipline, that is, when there are $k \geqslant 1$ customers in service, each one is served instantaneously at rate $1/k$;
	\item the sojourn times of impatient customers in queue (before possible service completion) are i.i.d.\ and exponentially distributed with mean $1/\theta$.
\end{itemize}

\begin{figure}[t]
% \scalebox{0.8}{\includegraphics[width=12cm, trim = 5cm 5.7cm 0cm 4cm,clip]
% {OpenLoopPSqueue.pdf}}
	\begin{center}
		\begin{tikzpicture}
			[block/.style={text width={width("Patient S-customers")+2pt}},
			align=center,
			scale=.99]
			\setlength{\tmpln}{25mm}
			\setlength{\tmplnn}{10mm}
			\setlength{\tmplnnn}{.1\tmpln}
			% queue
			\draw (0,0) -- ++ (\tmpln,0) -- ++ (0,\tmplnn) -- ++ (-\tmpln,0);
			\foreach \x in {0,1,...,5}{
				\draw [fill=red!50] (\tmpln-\x*\tmplnnn,0) rectangle ++ (-\tmplnnn, \tmplnn);}
			\foreach \x in {2,3,5}{
				\draw [fill=blue!50] (\tmpln-\x*\tmplnnn,0) rectangle ++ (-\tmplnnn, \tmplnn);}
				\draw [fill=green!50] (\tmpln+12pt,.5\tmplnn) circle (10pt);
				% S-customers
				\node [single arrow, draw=black, rotate=-30, fill=blue!50, minimum height=15mm] (fleche-S) at (-1.5, \tmplnn+5) {};
				\node [block] at ($(fleche-S)+(-2.5,0)$) {Patient S-customers (arrival rate $\alpha$)};
				% M-customers
				\node [single arrow, draw=black, rotate=30, fill=red!50, minimum height=15mm] (fleche-M) at (-1.5, -.5) {};
				\node [text width={width("Impatient M-customers")+2pt}] at ($(fleche-M)+(-2.5,0)$) {Impatient M-customers (arrival rate $\beta$)};
				% Departure
				\node [single arrow, draw=black, thick, minimum height=15mm] (fleche-D) at (\tmpln + 22pt + 10mm, .5 \tmplnn) {};
				\node [text width={width("Departure due to service ")+2pt}, anchor=south] at (\tmpln+25mm, \tmplnn) {Departure due to service completion (rate $(\mu n + \nu m)/(n+m)$)};
				% Mobility
				\node [single arrow, draw=black, thick, minimum height=15mm, dashed, rotate=-30] at (\tmpln + 22pt + 10mm, -.5\tmplnn) {};
				\node [text width={width("Departure due to impatience ")+2pt}, anchor=north] at (\tmpln+20mm, -1.4\tmplnn) {Departure due to impatience (rate $\theta m$)};
				% \node at ($(fleche-D)+(0,1)$) {Departure process};
				% \node [text width={width("(after either service completion")+2pt}, align=center] at ($(fleche-D)+(0,+1.5)$) {Departure process (after either service completion or impatience)};
				\node at (.5\tmpln,-.5) {PS service discipline};
		\end{tikzpicture}
	\end{center}
\caption{Multi-class PS queue with impatience.}
\label{PSqueue}
\end{figure}
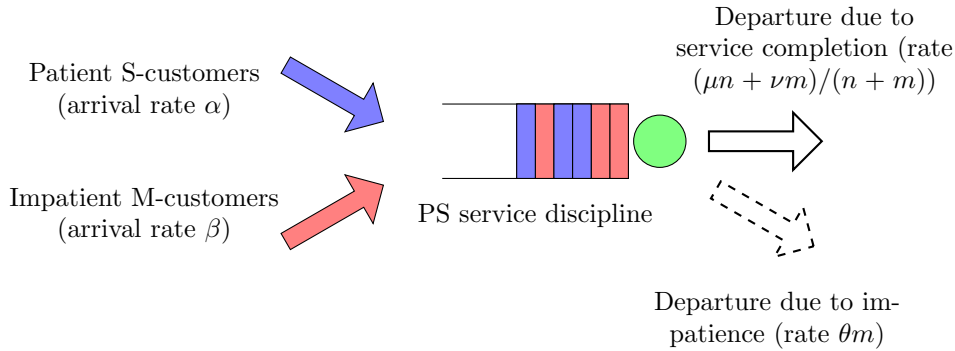

\noindent
This defines a birth-and-death process $(N,M) = ((N(t), M(t)), t \geqslant 0)$ with values in~$\mathbb{N}^2$ and whose infinitesimal generator~$\Omega$ is given by
\begin{multline*}
	\Omega(f)(n,m) = \alpha \left( f(n+1,m) - f(n,m) \right) + \beta \left( f(n,m+1) - f(n,m) \right)\\
	+ \frac{\mu n}{n+m} \left( f(n-1,m) - f(n,m) \right) + \left( \frac{\nu m}{n+m} + \theta m \right) \left( f(n,m-1) - f(n,m) \right)
\end{multline*}
for $f: \mathbb{N}^2 \to \mathbb{R}$ and $(n, m) \in \mathbb{N}^2$ (with the convention $0/0 = 0$). This process has a stationary distribution $(N(\infty), M(\infty))$ if and only if the stability condition
\begin{equation}
\varrho = \frac{\alpha}{\mu} < 1
\label{StabCond}
\end{equation}
holds~\cite[Sect.12.2, Prop.12.1]{SPRI18}, where $\varrho$ denotes the load of patient customers offered to the system. Note that this stability condition only involves the load of patient customers (through their arrival rate $\alpha$ and service requirement $\mu$) and not that of impatient ones, as the latter can always leave the system in a finite time whatever the system load.

%%%%%%%%%%%%%%%%%%%%
\subsection{Two time scales in the heavy traffic regime}
\label{TTS}
%%%%%%%%%%%%%%%%%%%%
In this queue, we are interested in the heavy traffic regime where $\beta$ tends to infinity, while the four other parameters $\alpha$, $\mu$, $\nu$ and $\theta$ remain fixed. We will consider $A = \beta / \theta$ as our scaling parameter and write $A \to \infty$ to mean that $\beta \to \infty$ with all other parameters kept fixed. In this regime, both processes $N$ and $M$ become of the order of $A$ but evolve on different time scales as can be observed when considering their fluid behavior.

As $A$ becomes large, $M$ becomes large and so departures are mostly due to the impatience term $\theta \cdot m$, given $M = m$ and $N = n$, since the service term $\nu m / (n+m)$ remains bounded. If this service term could be neglected, then $M$ would be equal to~$M'$, the $M/M/\infty$ queue length with input rate $\beta = A \theta$ and service rate $\theta$. As specified below, $M$ and $M'$ indeed behave very similarly in the considered heavy traffic regime. In fact, a simple coupling argument between $M$ and $M'$ makes it possible to transfer to $M$ the well-known heavy traffic behavior of $M'$, namely, to show that the process $(M(t)/A, t \geqslant 0)$ scaled only in space converges (weakly, in a functional sense) to the deterministic solution $(y(t), t \geqslant 0)$ to the ordinary differential equation (ODE)
$$
\frac{\mathrm{d}y}{\mathrm{d}t} = \theta - \theta \, y
$$
and that its stationary distribution $M(\infty) / A$ converges to the unique stable point
\begin{equation}
y^* = 1
\label{Defy*}
\end{equation}
of this ODE.
% This is precisely the asymptotic behavior of $M'$ in this regime.

On the other hand, arrival and service rates of $N$ remain bounded: they are respectively equal to $\alpha$ and $\mu n / (n + m) \in [0,\mu]$. As defined by this service rate, component $N$ needs to become commensurate with component $M$ in order to obtain some service and so it will also live on the $O(A)$ space scale. But since its arrival rate is bounded, it needs a time of order $O(A)$ to reach such values and it is indeed on this time scale that it evolves. On this time scale, however, $M$ evolves very rapidly and so an averaging behavior is to be expected, whereby $N$ and $M$ would interact through the mean value of $M$ which, as argued above, is close to $A$. In other words, the asymptotic behavior of $N$ is expected to be close to that of $N'$, the length of the single-server PS queue with $A$ permanent customers. In fact, standard methods could be used to prove that $N$ and $N'$ have the same fluid limit; specifically, the process $(N(At)/A, t \geqslant 0)$ scaled both in time and space converges to the deterministic solution $(x(t), t \geqslant 0)$ to the ODE 
$$
\frac{\mathrm{d}x}{\mathrm{d}t} = \alpha - \mu \, \frac{x}{x+1},
$$
and its stationary distribution $N(\infty)/A$ converges to the unique stable point
\begin{equation}
x^* = \frac{\varrho}{1-\varrho}
\label{Defx*}
\end{equation}
of this ODE, with again $\varrho = \alpha/\mu < 1$.

In other words, in the heavy traffic regime when $A \to \infty$, the fluid behavior of $(N,M)$ is the same as that of $(N', M')$ and the main goal of this paper is to investigate to which extent this approximation holds in a Large Deviations setting.

%%%%%%%%%%%%%%%%%%%%%%%%%%%
\subsection{Main results}
%%%%%%%%%%%%%%%%%%%%%%%%%%%
In order to emphasize the dependency with respect to the scaling parameter $A$, let us denote by $\pmb{\Pi}_A$ the stationary distribution of $(N,M)$ when $\beta / \theta = A$ (recall that we let $A \to \infty$ while the four parameters $\alpha, \mu, \nu$ and $\theta$ remain fixed). It follows from the above discussion that the mass of distribution $\pmb{\Pi}_A$ is essentially concentrated around $(Ax^*, Ay^*)$ in the sense that $\pmb{\Pi}_A([A \underline x, A \bar x] \times [A \underline y, A \bar y]) \to 1$ when $A \to \infty$, for any $\underline x < x^* < \bar x$ and $\underline y < y^* < \bar y$. This regime therefore defines a Large Deviations setting for $\pmb{\Pi}_A$, whereby probabilities $\pmb{\Pi}_A(Ax,Ay)$ decrease exponentially for increasing $A$ and fixed 
$x \geqslant 0$, 
$y \geqslant 0$. 

The main result of the present paper is to establish sharp local asymptotics using the singular perturbation method, as discussed in more detail in Section~\ref{sec:asymptotics} below. In this framework, it is admitted that an expansion of the form
\begin{multline}
\pmb{\Pi}_A(Ax, Ay) = \frac{1}{2\pi A} \; \times 
\\
\exp \left [ -A \cdot H(x, y) - h_0(x, y) - 
\frac{h_1(x, y)}{A} - 
\frac{h_2(x, y)}{A^2} + O\left( \frac{1}{A^3} \right) \right ], \quad 
x, \; y > 0,
\label{AS0}
\end{multline}
exists for functions $H$ and $h_0$, $h_1$, $h_2$ satisfying some specific smoothness assumptions; these functions are then successively determined via the Kolmogorov equations. Note that $H$ in expansion (\ref{AS0}) is the usual decay function of the Large Deviations theory, defined by
$$
H(x,y) = - \, \lim_{A \to \infty} \frac{1}{A} \log \pmb{\Pi}_A(Ax, Ay).
$$
% This is the approach that we unfold here, and that leads to the following main result. This 
Our main result involves the functions $\Phi$, $\Psi$ and $g$ that will appear repeatedly in the sequel, and which are respectively defined by
\begin{align}
\Phi(x) & = x \log \left ( \frac{x}{\varrho} \right ) -(x+1)\log(x+1) - \log(1-\varrho), \qquad x \geqslant 0,
\label{eq:Phi}
\\
\Psi(y) & = y \log y - y + 1, \qquad y \geqslant 0,
\label{eq:Psi}
\end{align}
and
\[ g(x,y) = (1-\varrho) \sqrt{\frac{x+1}{x \, y}} \left ( \frac{x+1}{x+y} \right )^{\nu/\theta} \exp \left [ \frac{\mu}{\theta}(1-\varrho) \left ( \frac{x-x^*}{x+1} \right ) \log \left ( \frac{x+1}{x+y} \right ) \right ] \]
for $x, \; y > 0$ (recall that $x^*$ and $y^*$ have been defined in (\ref{Defx*}) and (\ref{Defy*})).

\begin{theorem}\label{T2}
	 Beside stability condition $\varrho < 1$, assume further that an asymptotic expansion of the form~\eqref{AS0} exists and satisfies the following smoothness conditions:
	\begin{enumerate}
		\item the functions $H$, $h_0$, $h_1$ and $h_2$ are respectively of class $\mathscr{C}^3$, $\mathscr{C}^2$, $\mathscr{C}^1$ and $\mathscr{C}^0$ in the open quarter-plane $\mathbb{R}^{+*} \times \mathbb{R}^{+*}$;
		\item the decay function $H$ is non negative, continuous over the closed quarter plane $\mathbb{R}^+ \times \mathbb{R}^+$, and satisfies $H(x^*,y^*) = 0$.
	\end{enumerate}
	Then as $A \to \infty$, we have
	\begin{equation}
	\pmb{\Pi}_A(Ax,Ay) \sim \frac{g(x,y)}{2\pi A} \, 
	e^{-A \cdot (\Phi(x)+\Psi(y))}
	\label{AsymptPA}
	\end{equation}
	for any $x, \; y > 0$.
\end{theorem}

The assumption on the continuity of $H$ over $\mathbb{R}^+ \times \mathbb{R}^+$ is motivated by the fact that these properties hold in the case when a large deviations principle (LDP) exists (this results from the lower semi-continuity of $H$~\cite[Chap.7, Sect.6]{Frei12}, together with the existence of an attained infimum for the action functional on any closed subset~\cite[p.81]{Frei12}). Focusing on the decay function $H$, Theorem \ref{T2} has the following consequence.

\begin{theorem}\label{T1}
	Under the same assumptions as that of Theorem~\ref{T2}, the decay rate $H$ of distribution $\pmb{\Pi}_A$ equals the sum
	$$
	H(x,y) = \Phi(x) + \Psi(y)
	$$
	for all $x,\; y \geqslant 0$.
\end{theorem}

\begin{figure}[h]
\scalebox{0.75}{\includegraphics[width=9cm, trim = 0cm 5cm 8cm 3.5cm,clip]
{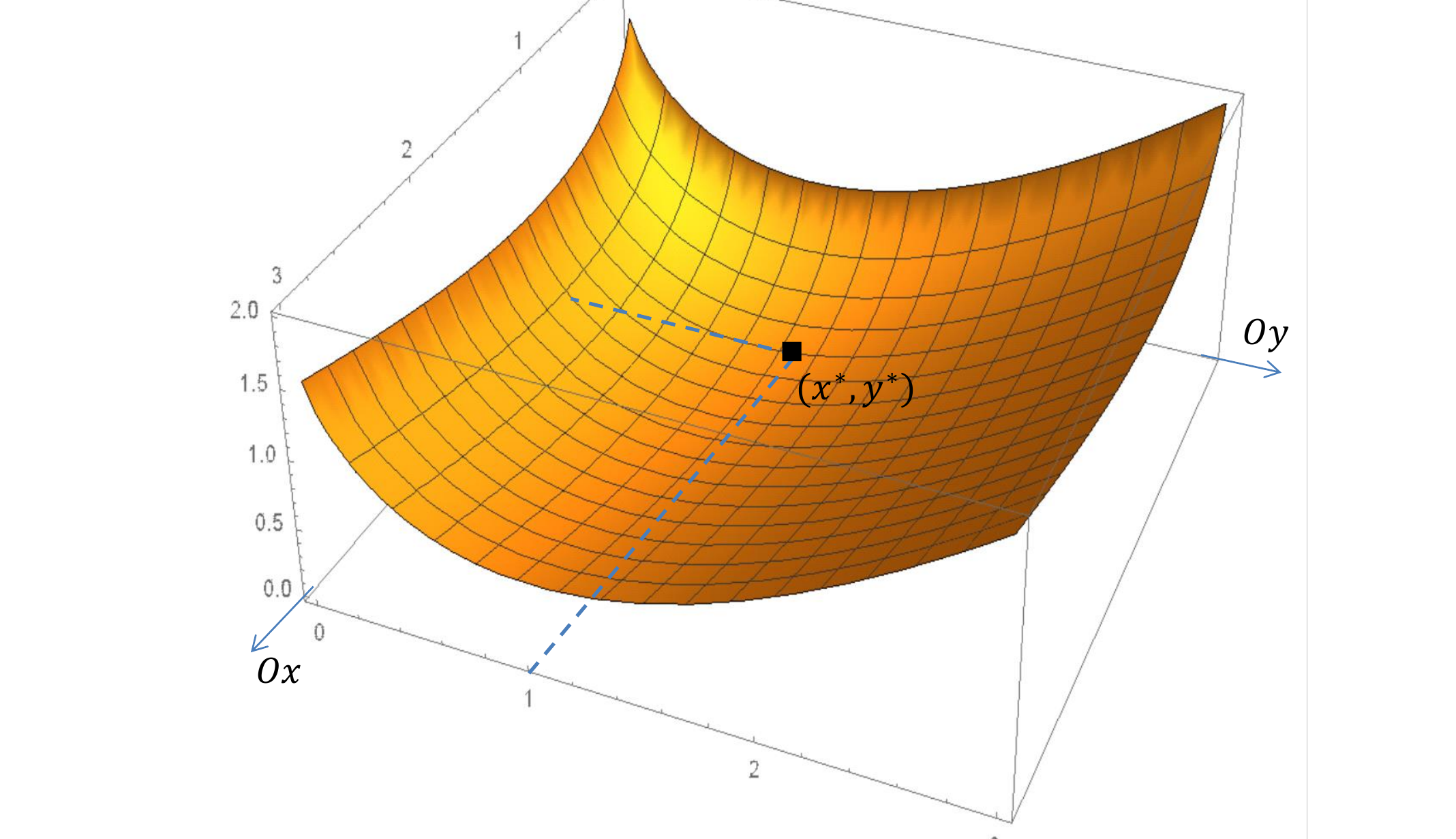}}
\scalebox{0.72}{\includegraphics[width=8cm, trim = 0cm 5cm 12cm 4cm,clip]
{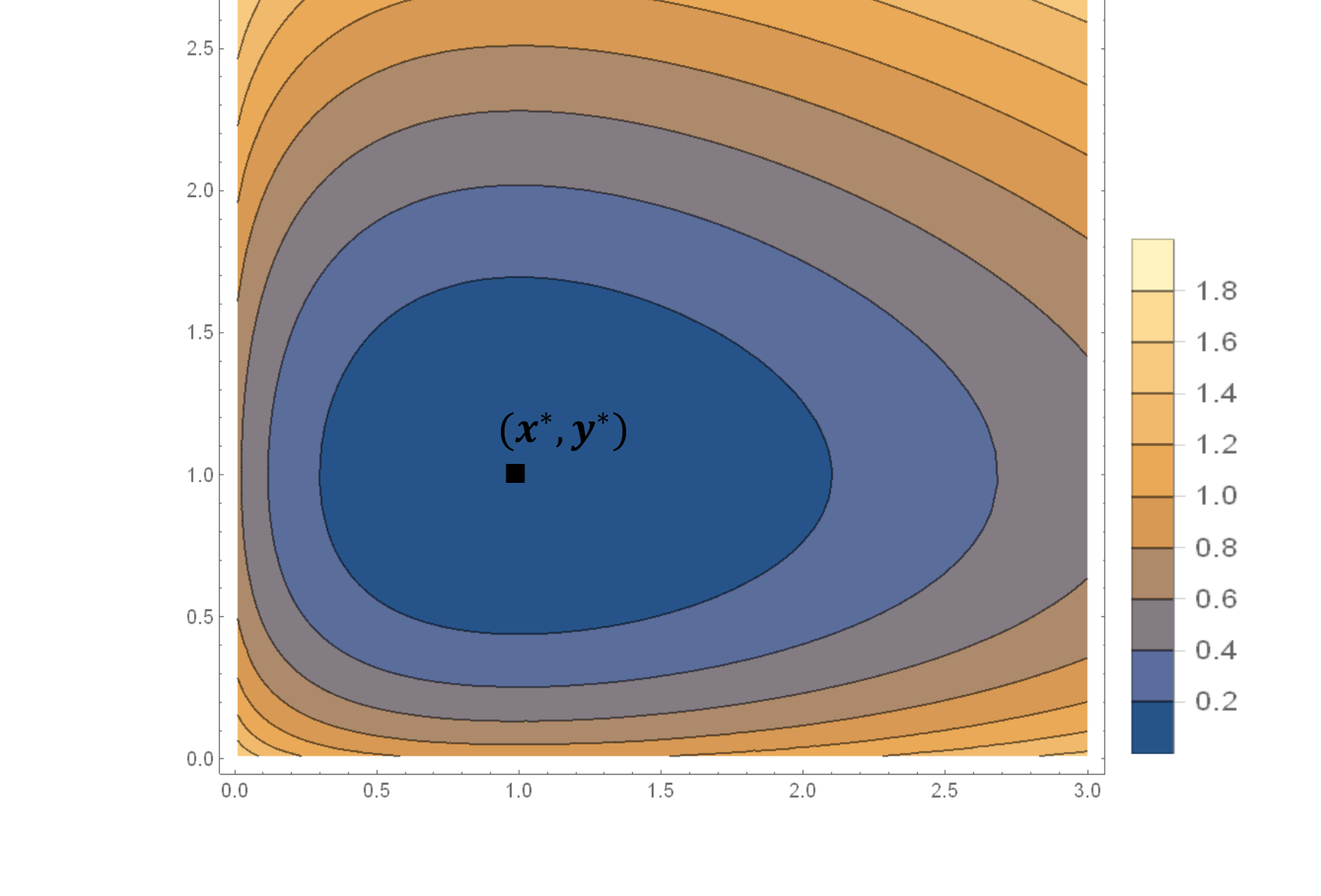}}
\caption{Surface $z = H(x,y)$ and level curves $H(x,y) = \textnormal{\textrm{constant}}$.}
\label{Fig12}
\end{figure}

\noindent
For illustration (see Figure~\ref{Fig12}), the convex surface $z = H(x,y)$ in the $(x,y,z)$-space is plotted for $\varrho = 0.5$. We then have $(x^*,y^*) = (1,1)$ and, in particular, $H(0,0) \approx 1.69$, $H(3,0) \approx 1.52$, $H(0,3) \approx 1.99$. The level curves $H(x,y) = \textnormal{\textrm{constant}}$ in the positive quadrant are also depicted.

Theorem~\ref{T1} thus asserts that distribution $\pmb{\Pi}_A$ is asymptotically the product of two marginal distributions in the logarithmic order. Actually, the components $\Phi$ and $\Psi$ that appear are exactly those of the processes $N'$ and $M'$ introduced earlier, that is, $\Phi$ is the decay rate of the single-server PS queue with input rate $\alpha$, $A$ permanent customers and service rate $\mu$ (this will be proved in Appendix~\ref{PL1bis}), and $\Psi$ is the decay rate of the $M/M/\infty$ queue with input rate $A \theta$ and service rate $\theta$~\cite[Chap.5, p.160]{Frei12}. This result therefore shows that the approximation $(N, M) \approx (N', M')$ remains accurate in the logarithmic order for large deviations. However, Theorem~\ref{T2} shows that this approximation breaks down in the usual, say $O(1)$, order because function $g$ in~\eqref{AsymptPA} does not factorize into the product of two functions of $x$ and $y$. 

The next result shows that this independence property in the logarithmic order is enough to imply independence of centered and scaled stationary distributions.

\begin{theorem} 
Under the same assumptions as that of Theorem~\ref{T2}, the centered pair
\[ (\xi_A,\eta_A) = \sqrt{A} \left ( \frac{N(\infty)}{A} - x^*, 
\frac{M(\infty)}{A} - y^* \right) \]
converges weakly as $A \to \infty$ towards the centered Gaussian variable $(\xi,\eta)$ with covariance structure
\[ 
\mathbb{E}(\xi^2) = \frac{\varrho}{(1-\varrho)^2}, \quad 
\mathbb{E}(\eta^2) = 1, \quad \mathbb{E}(\xi \eta) = 0.
\]
Moreover, we have
$$
\mathbb{E}(N(\infty)) \sim A \, x^*, \qquad \mathbb{E}(M(\infty)) \sim A
$$
for large $A$.
\label{T0}
\end{theorem}

Theorem~\ref{T0} implies, in particular, that the scaled pair $(N(\infty)/A, M(\infty)/A)$ converges weakly to the deterministic point $(x^*,y^*)$, as was alluded to before. Besides, the asymptotic distribution of $(\xi,\eta)$ has zero covariance, so that its components are asymptotically independent, although $N(\infty)$ and $M(\infty)$ are dependent for finite~$A$. In fact, $(\xi, \eta)$ is the limit of the centered and scaled stationary distribution of $(N', M')$, showing that the approximation $(N, M) \approx 
(N', M')$ still holds for fluctuations of stationary distributions around their deterministic limits
% . Note finally that $\xi$ (resp. $\eta$) is still the limit of the normalized fluctuation $\sqrt{A}(N'(\infty)/A - x^*)$ (resp. of $\sqrt{A}(M'(\infty)/A - y^*)$) for increasing $A$ 
(this is verified in Appendix \ref{PL1bis}, Remark \ref{Rs}, for the variable $\xi$; on the other hand, this readily follows from the Gaussian approximation of the Poisson distribution of the $M/M/\infty$ queue for the variable $\eta$).

%%%%%%%%%%%%%%%%%%%%%%%%%%
\subsection{Application to mobile networks}
%%%%%%%%%%%%%%%%%%%%%%%%%%

Numerous models of multi-class PS queues with impatience have been investigated in the queueing literature. The single-class PS queue with impatience has been early dealt with to derive asymptotics for the stationary queue distribution~\cite{Coff94}. In the context of radio communication networks, the multi-class case when all classes are impatient has been addressed for the control of early customer departure in the overload regime~\cite{GUI15}. More recently, this multi-class queue has been invoked for the performance of radio networks when accounting for spatial mobility~\cite{SPRI18, POAS16, SIM20}. In this context, impatience is used to model mobility, as both impatience and mobility make customers leave the system independently of the service received.

This last stream of results is actually one of our motivation for investigating Theorem~\ref{T1}.
\begin{figure}[t]
% \scalebox{0.8}{\includegraphics[width=12cm, trim = 1.5cm 1.5cm 0cm 5cm,clip]
% {OpenLoopPSqueueBIS0.pdf}}
	\begin{center}
		\begin{tikzpicture}[block/.style={
			text width={width("Patient S-customers")+2pt}},
			align=center]
			\setlength{\tmpln}{25mm}
			\setlength{\tmplnn}{10mm}
			\setlength{\tmplnnn}{.1\tmpln}
			% queue
			\draw (0,0) -- ++ (\tmpln,0) -- ++ (0,\tmplnn) -- ++ (-\tmpln,0);
			\foreach \x in {0,1,...,5}{
				\draw [fill=red!50] (\tmpln-\x*\tmplnnn,0) rectangle ++ (-\tmplnnn, \tmplnn);}
			\foreach \x in {2,3,5}{
				\draw [fill=blue!50] (\tmpln-\x*\tmplnnn,0) rectangle ++ (-\tmplnnn, \tmplnn);}
				\draw [fill=green!50] (\tmpln+12pt,.5\tmplnn) circle (10pt) coordinate (server);
				% S-customers
				\node [single arrow, draw=black, rotate=-30, fill=blue!50, minimum height=15mm] (fleche-S) at (-1.5, \tmplnn+5) {};
				\node [block] at ($(fleche-S)+(-2.5,0)$) {S-customers, exogenous flow};
				% M-customers
				\node [single arrow, draw=black, rotate=00, fill=red!50, minimum height=15mm] (fleche-M) at (-1.5, -0) {};
				\node [text width={width("Patient S-customers")+2pt}] at ($(fleche-M)+(-2.5,-.2)$) {M-customers, exogenous flow at rate $\beta_\texttt{ex}$};
				% \node [single arrow, draw=black, dashed, rotate=-30, fill=red!50, minimum height=15mm] (fleche-loop) at (\tmpln+45pt, -.4\tmplnn) {};
				% Departure
				\node [single arrow, draw=black, thick, minimum height=15mm] (fleche-D) at (\tmpln + 22pt + 10mm, .5 \tmplnn) {};
				% \node at ($(fleche-D)+(0,1)$) {Departure process};
				\node [text width={width("(after either service completion")+2pt}, align=center] at ($(fleche-D)+(0,+1)$) {Output flow after service completion};
				\node at (.5\tmpln,-.5) {PS queue};
				% \pgfsetarrowsend{triangle 90}
				% Mobility
				\node [single arrow, draw=black, thick, minimum height=10mm, dashed, rotate=-30] at (\tmpln + 10pt + 10mm, -.5\tmplnn) {};
				% \node [text width={width("Depature due to impatience ")+2pt}, anchor=north] at (\tmpln+20mm, -1.4\tmplnn) {Depature due to impatience (rate $\theta m$)};
				
				\draw [line width=5pt, draw=red!50, >=stealth, ->] (\tmpln+50pt, -\tmplnn) to [out=-50,in=-130] (-.5,-.5);
				\node [text width={width("M-customers, feedback")+10pt}, align=center] at (.65\tmpln,-80pt) {M-customers, feedback flow due to impatience at rate $\beta_\net = \theta \mathbb{E}{M(\infty)}$};
		\end{tikzpicture}
	\end{center}
\caption{Multi-class closed-loop PS queue.}
\label{PSqueueBIS0}
\end{figure}
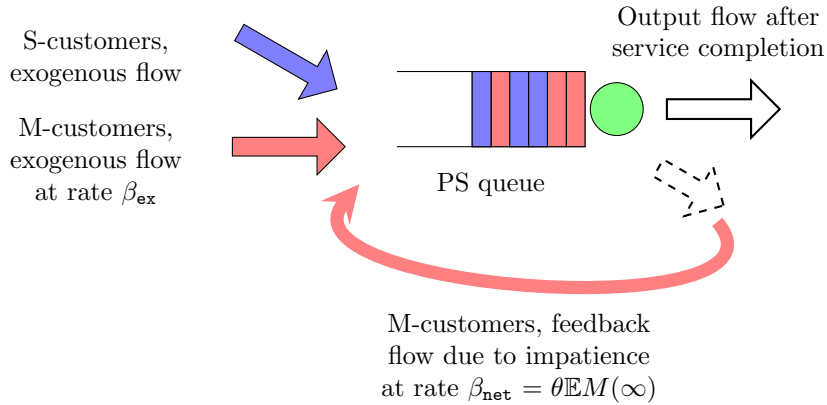
These papers consider the process $(N,M)$ above, with $N$ the number of Static (patient) customers, and $M$ the number of Moving (impatient) customers in the considered radio cell: a departure of an $M$-customer is thus either due to a service completion, or to a spatial movement to another neighboring cell of the network, the latter happening in stationarity at rate $\theta \; \mathbb{E}(M(\infty))$.

In order to account for possible reverse movements of users to the considered cell in the network but outside the cell, the authors of~\cite{SPRI18, POAS16, SIM20} consider the so-called \textit{closed-loop} Processor-Sharing queue (see Figure~\ref{PSqueueBIS0}). In the latter, the arrival rate of $M$-customers is decomposed as
$$
\beta = \beta_{\texttt{ex}} + \beta_{\net},
$$
$\beta_{\texttt{ex}}$ representing the rate of exogenous arrivals and $\beta_{\net}$ the rate of arrivals within the network. The rate $\beta_\texttt{ex}$ is fixed, while the rate $\beta_\net$ is obtained by imposing a balance condition. In fact, the authors consider the case of a balanced cell where the movements of mobile users within the network from and to the cell balance each other, that is, $\beta_\net$ is equal to the rate $\theta \; \mathbb{E}(M(\infty))$ of customers moving out of the cell. This balance condition is captured by the equation
\begin{equation}
\theta \cdot \mathbb{E}(M(\infty)) = \beta_{\net}. 
\label{FPEqu}
\end{equation}
% Both terms are due to mobility: the left-hand side is the outgoing rate from the cell to the rest of the network, and the right-hand side is the reverse rate from the network to the cell. Equating the two amounts to assuming an equal flow of customers leaving and entering the cell because of mobility. Note that
Since $\mathbb{E}(M(\infty))$ is itself a function of $\beta_{\net}$, (\ref{FPEqu}) is a Fixed-Point equation. It has been proved~\cite[Proposition 3.1]{SPRI18} that this Fixed-Point equation has a unique solution if and only if 
$$
\varrho_\tot := \frac{\alpha}{\mu} + \frac{\beta_{\texttt{ex}}}{\nu} < 1, 
$$
meaning that the total load imposed by exogenous arrivals is smaller than the cell capacity. When it is enforced, this defines a Markov process $(N_\mob, M_\mob)$ which is a particular case of the above $(N,M)$ process with a parameter $\beta$ specifically chosen as an implicit function of other parameters $\alpha, \beta_\texttt{ex}, \mu$ and $\nu$, that is, $\beta = \beta_\texttt{ex} + \beta_\net$ with $\beta_\net$ determined by Fixed-Point equation (\ref{FPEqu}).

In~\cite{SIM20}, this Markov process $(N_\mob, M_\mob)$ is studied in the heavy traffic regime $\varrho_\tot \uparrow 1$: this makes the rate $\beta_\net$ of inner movements grow large and it thus amounts to studying the $(N,M)$ process in the regime $A \to \infty$. It is proved there, in particular, that the stationary distribution remarkably grows as the logarithm of $1/(1-\varrho_\tot)$, a very peculiar result in sharp contrast with the usual $1/(1-\varrho_\tot)$ growth in heavy traffic. More precisely, the authors show that the random sequence $(N_\mob(\infty), M_\mob(\infty))/\log(1/(1-\varrho_\tot))$ is tight when $\varrho_\tot \uparrow 1$, that any accumulation point is larger than $(x^*, y^*)$ and they conjecture that this lower bound is actually the exact limit. As argued in~\cite{SIM20}, proving this requires to prove that 
$$
-\frac{1}{A} \log \pmb{\Pi}_A(0,0) \longrightarrow H(0,0) = 1-\log(1-\varrho)
$$
when $A \to \infty$, which is a direct consequence of Theorem~\ref{T1} in the framework of the present singular perturbation setting. %\textbf{\textcolor{red}{Not clear: Theorem~\ref{T1} does not clearly say that $\mathbb{P}(N=M=0) \sim e^{-A H(0,0)}$, i.e., that is holds for $x=y=0$}}

It is proved in~\cite{SIM20} that $\mathbb{P}(N_\mob(\infty) = M_\mob(\infty) = 0) = 1 - \varrho_\tot$ so that as $\varrho_\tot \uparrow 1$, we have $\beta \to \infty$ in such a way that $A \sim - \log(1-\varrho_\tot) / H(0,0)$. A direct application of Theorem~\ref{T1} to the 
$(N_\mob, M_\mob)$ process then enables us to state the following.

\begin{theorem}
	Suppose $\varrho < 1$ and that the assumptions of Theorem~\ref{T2} hold. We further let 
	$$
	A_\mob = - \, \frac{\log(1-\varrho_\tot) }{H(0,0)}
	$$
	with $H(0,0) = 1 - \log(1-\varrho)$. As $\varrho_\tot \uparrow 1$, the centered pair
% \textbf{For the closed-loop PS queue, denote by}
%
% \textbf{$\varrho$ the load offered by the exogenous flow of \textit{S}-customers,}
%
% \textbf{$\varrho_\tot$ the total load offered by exogenous flows of both \textit{S} and \textit{M}-customers. Keeping $\varrho$ fixed with $\varrho < \varrho_\tot < 1$, define constant $c^*$ by}
% \begin{equation}
% \frac{1}{c^*} = 1 - \log(1-\varrho)
% \label{Defc0}
% \end{equation}
% \textbf{and the parameter $A = -c^* \, \log(1-\varrho_\tot)$ , $\varrho_\tot < 1$.}
%
% \textbf{When $\varrho_\tot \uparrow 1$, the scaled queue occupancy $(N_\mob(\infty),M_\mob(\infty))/A$ of the closed-loop PS queue converges weakly to $(x^*,y^*)$ and}
$$
\sqrt{A_\mob} \, \left ( \frac{N_\mob(\infty)}{A_\mob} - x^*, \frac{M_\mob(\infty)}{A_\mob} - y^* \right )
$$
converges weakly to the same Gaussian variable $(\xi, \eta)$ as that of Theorem~\ref{T0}. Moreover, the mean queue occupancies grow logarithmically as
$$
\mathbb{E}(N_\mob(\infty)) \sim 
A_\mob \, x^* 
%- \, \frac{\varrho}{(1-\varrho)H(0,0)} \log(1-\varrho_\tot), 
\qquad \text{ and } \qquad
\mathbb{E}(M_\mob(\infty)) \sim 
A_\mob
%- \, \frac{\log(1-\varrho_\tot)}{H(0,0)}
$$
when $\varrho_\tot \uparrow 1$, with again $x^*= \varrho/(1-\varrho)$ and $y^* = 1$.
%\commentaire{ok for $E(M)$ thanks to the fixed point equation, but why does this hold for $E(N)$?}
\label{T4}
\end{theorem}

% For the closed-loop queue, Theorem~\ref{T4} therefore entails that the mean queue occupancies grow logarithmically as
% $$
% \mathbb{E}(N_\mob(\infty)) \sim - \frac{c^* \varrho}{1-\varrho} \log(1-\varrho_\tot),
% \qquad
% \mathbb{E}(M_\mob(\infty)) \sim \displaystyle - c^* \log(1-\varrho_\tot)
% $$
% in the heavy load regime when $\varrho_\tot \uparrow 1$, where constant $c^*$ is the inverse of the logarithmic decay rate for the empty-queue probability, as given by~\eqref{Defc0}.
%
% Finally,

The latter estimates of the mean queue occupancies enable us to derive asymptotics for the average throughput of each customer class. Seeing the workload brought by each arriving customer as a data volume to be transferred through a communication link (server) with total transmission capacity $C$, the mean throughput can be defined as the ratio of the mean volume of transferred data to the mean transfer time of a given customer~\cite[Section 2.1]{POAS16}. Normalizing the server capacity $C$ to unity and using the general expressions of~\cite[Prop.2.2]{POAS16}, the efficient throughputs $\gamma$ and $\Gamma$ of class \textit{S} (Static) and \textit{M} (Moving) customer flows can then be readily expressed by
$$
\gamma = \frac{\varrho}{\mathbb{E}(N_\mob(\infty))}, \qquad 
\Gamma = \frac{1}{\mathbb{E}(M_\mob(\infty))}
\left ( \varrho_\tot- \varrho + \frac{\beta_\net}{\nu} \right ) - \frac{\theta}{\nu},
$$
respectively, where rate $\beta_\net$ is defined by (\ref{FPEqu}). %\commentaire{Check the definition of $\beta_\net$ here, it is ambiguous and needs to be checked and clarified} 
As $\varrho_\tot \uparrow 1$, the estimates of 
$\mathbb{E}(N_\mob(\infty))$ and $\mathbb{E}(M_\mob(\infty))$ provided by Theorem \ref{T4} then yield
\begin{equation}
\gamma \sim -(1- \log(1-\varrho)) \, \frac{1-\varrho}{\log(1-\varrho_\tot)}, \qquad 
\Gamma \sim -(1- \log(1-\varrho)) \, \frac{\varrho_\tot-\varrho}{\log(1-\varrho_\tot)}
\label{Throughps}
\end{equation}
for each customer class of the closed-loop queue.

\subsection{Organization of paper}
%%%%%%%%%%%%%%%%%%%%%%%%%%%%

Before presenting the proofs of the latter results, we first discuss in Section~\ref{sec:asymptotics} their relevance compared to the current literature on both Large Deviations and Singular Perturbation methods. Section~\ref{Sec:SCH} contains preliminary technical results. Although Theorem~\ref{T1} above was claimed as a consequence of Theorem~\ref{T2}, the proof proceeds by first proving Theorem~\ref{T1} in Section~\ref{Sec:DRF}, and then iterating the argument to prove Theorem~\ref{T2} in Section~\ref{Sec:FAE}. Section~\ref{Sec:FAE} also presents a direct Corollary to Theorem~\ref{T2} concerning the asymptotic behavior of the marginal distributions of $N(\infty)$ and $M(\infty)$ (Corollary~\ref{C1}). The proof of Theorem~\ref{T0} is then given in Section~\ref{Sec:EOB}; it essentially relies on the asymptotics that Theorem~\ref{T2} enables us to obtain for the generating function of distribution $\pmb{\Pi}_A$. Appendix~\ref{PL1bis} establishes that function $\Phi$ is the decay rate of the single-server PS-queue with $A$ permanent customers; Appendices~\ref{PL0} and~\ref{PL1} provide the proofs of two intermediate results that intervene in the proof of Theorem~\ref{T0}.

% \section{Motivation} \label{sec:motivation}

%%%%%%%%%%%%%%%%%%%%%%%%%%%%%%
\section{Asymptotics of stationary distributions}
\label{sec:asymptotics}
%%%%%%%%%%%%%%%%%%%%%%%%%%%%%%

Prior to proceeding to the detailed proofs of our main results, we first review previous works addressing asymptotics for the stationary distribution of Markov jump processes.

%%%%%%%%%%%%%%%%%%%%
\subsection{Large Deviations Principles} \label{sub:LDP}
%%%%%%%%%%%%%%%%%%%%

Consider a scaled jump process 
$\mathbf{Z}_A$ in some subset of the lattice $\mathbb{Z}^d/A$, $d \geqslant 2$. The scaling applied to $\mathbf{Z}_A$ is said \textit{regular} if all transition rates are proportional to parameter $A$. Assume then that an LDP can be stated for $\mathbf{Z}_A$, with an action functional $S_T$ defined on the metric space $\mathscr{C}_{T}(\mathbb{R}^d)$ of continuous 
$\mathbb{R}^d$-valued functions on interval $[0,T]$, $T \geqslant 0$. If process 
$\mathbf{Z}_A$ has a stationary distribution $\pmb{\Pi}_A$, its decay rate
\begin{equation}
H(\mathbf{z}) = - \lim_{A \uparrow +\infty} \frac{1}{A} \cdot 
\log \pmb{\Pi}_A(A\mathbf{z}), \qquad 
\mathbf{z} = (z_1,\ldots,z_d) \in \mathbb{R}^d,
\label{DecayH}
\end{equation}
is then obtained~\cite[Chap.5, 6]{Frei12} by minimizing functionals $S_T$, $T \geqslant 0$, on the whole union 
$\bigcup_{T \geqslant 0} \mathscr{C}_{T}(\mathbb{R}^d)$. 

The scaling presently envisaged for process $(N,M)$, however, is not regular since only 
$\beta = O(A)$ grows to infinity while $\alpha$ is kept fixed. This amounts to squeezing the time scale of the impatient customers arrival process, while keeping the initial time scale for the patient customers arrival flow. For this \textit{singular} scaling, $N$ is thus seen as a slow process driven by the fast variations of $M$.

Similar settings have been investigated in previous work, but none seems to directly apply to our problem. Given a homogeneous Markov chain $\mathbf{Y}$ with finite state space 
$\Gamma \subset \mathbb{N}$, consider the pair 
$\mathbf{Z}_A = (\mathbf{X}_A,\mathbf{Y}_A)$ where the fast process 
$\mathbf{Y}_A(t) = \mathbf{Y}(At)$, $t \geqslant 0$, drives the slow process 
$\mathbf{X}_A$ via the differential equation 
\begin{equation}
\frac{\mathrm{d}\mathbf{X}_A}{\mathrm{d}t}(t) = 
\mathbf{b}(\mathbf{X}_A(t),\mathbf{Y}(At)), 
\qquad t \geqslant 0,
\label{EquDiffA}
\end{equation}
for a drift 
$\mathbf{b}:\mathbb{R}^{d-1} \times \Gamma \rightarrow \mathbb{R}^{d-1}$. Then:
\begin{itemize}
	\item an LDP can be stated~\cite[Chap.7, Section 4]{Frei12} for the slow component $\mathbf{X}_A$ of $\mathbf{Z}_A = (\mathbf{X}_A,\mathbf{Y}_A)$, with an action functional $S_T$ defined on space $\mathscr{C}_{T}(\mathbb{R}^{d-1})$; 
	\item consider further the set $\mathscr{L}_{0,T}^\Gamma$ of mappings 
$\mathfrak{p}:(t,y) \in [0,T] \times \Gamma \mapsto \mathfrak{p}(t,y)$ such that 
$\mathfrak{p}(\cdot,y)$ is Borelian on $[0,T]$ for each $y \in \Gamma$, and the vector $(\mathfrak{p}(t,y))_{y \in \Gamma}$ is a probability on $\Gamma$ for each $t \in [0,T]$. Let then the process $\mathfrak{P}_A$ be the random element of $\mathscr{L}_{0,T}^\Gamma$ defined by 
\[ \mathfrak{P}_A(t,y) = \mathbf{1}_{Y_A(t) = y}, \qquad t \in [0,T], \; 
y \in \Gamma. \]
An LDP for the pair $(\mathbf{X}_A,\mathfrak{P}_A)$ is then stated in ~\cite[Theorem 2.3]{Fag10}, with an action functional $\mathfrak{S}_T$ now defined on the product space $\mathscr{C}_{T}(\mathbb{R}^{d-1}) \times \mathscr{L}_{0,T}^\Gamma$.
\end{itemize}

%\commentaire{I added two references to the case of diffusion processes}

The case where $\mathbf{X}_A$ is a diffusion process has also received attention. In~\cite{Puhalskii16:0}, a general LDP is derived when both $\mathbf{X}_A$ and $\mathbf{Y}_A$ are diffusion processes while, closer to our case, \cite{Huang16:0} considers the case where $\mathbf{X}_A$ is a diffusion process and $\mathbf{Y}_A$ a finite-state space Markov chain. To our knowledge, however, no general LDP is known in the case when the slow component $\mathbf{X}_A$ is itself a Markov chain depending on the evolution of the fast driving chain $\mathbf{Y}_A$, both evolving with increments of order $O(1/A)$, all the more since all previous results assume the finiteness of the state space $\Gamma$ of the fast process $\mathbf{Y}$, which assumption fails for the process $M$ presently considered.

% As mentioned in the Introduction, Equ.(\ref{C0Intr}), a lower bound for the decay rate $1/c^*$ has been obtained for the \textit{closed-loop} PS queue~\cite[Theorem 2.5]{SIM20}. Without relying on a full LDP, this lower bound is alternatively derived by a tightness property of the sequence of stationary queue occupancies $\mathbf{Z}_A = (N_\mob(\infty),M_\mob(\infty))/A$ when $A \uparrow +\infty$. This property is itself obtained through a probabilistic coupling between $\mathbf{Z}_A$ and a simpler process $\mathbf{Z}_A'$
% for which the limit can be characterized. So far given as a conjecture~\cite[Conjecture 5.1]{SIM20}, the
% exact value of decay rate $1/c^*$ will be the object of Theorems~\ref{T1} and~\ref{T4} below.

%%%%%%%%%%%%%%%%%
\subsection{Sharp Asymptotics via singular perturbation methods} 
%%%%%%%%%%%%%%%%%

LDP's concern the asymptotic behavior of stationary distributions on the logarithmic scale. In order to derive sharp (that is, not only logarithmic) asymptotics, we will now invoke singular perturbation methods. 
%
%
% A probabilistic LDP being available or not, an analytical approach can be alternatively applied to derive sharp (and not only logarithmic) asymptotics of $\pmb{\Pi}_A$. This approach belongs to the domain of the so-called Singular Perturbation methods elaborated in various application fields.
These methods have been justified for specific classes of problems: 

- both a classification and rigorous foundation are established in~\cite[Chap.6]{Eck79} for some classical families of partial differential equations;

- the present case of jump processes has been considered in~\cite[Chap.4, 6]{Yin13} where asymptotics of the solutions of transient backward or forward Kolmogorov equations at finite time $t$ are stated, but for a regular scaling only (in a different meaning to that introduced in Section \ref{TTS} above, the two-time scales in \cite{Yin13} refer to either small $t = O(\varepsilon)$ or large $t = O(1/\varepsilon)$);

- asymptotic expansion for Laplace transforms have also been proven in the following context \cite{Fle89}. Consider a Markov process $\mathbf{Z}_A$ in $\mathbb{R}^d$, moving with a deterministic drift $\mathbf{b}$ and perturbed by a jump process with jump rates $O(A)$ and increments $O(1/A)$. Given a function $f \in \mathscr{C}^\infty(\mathbb{R}^d;\mathbb{R})$, let
$$
\mathrm{F}_A(\mathbf{z},t) = \mathbb{E} 
\left ( e^{-A \cdot f(\mathbf{Z}_A(T))} 
\, \lvert \, \mathbf{Z}_A(t) = \mathbf{z} \right ), 
\qquad \mathbf{z} \in \mathbb{R}^d, \; t \in [0,T].
$$
Provided that the drift $\mathbf{b}$ belongs to $\mathscr{C}^\infty(\mathbb{R}^d;\mathbb{R}^d)$ and the transition distribution of $\mathbf{Z}_A$ satisfies boundedness and non degeneracy conditions, it is then shown \cite[Theorem 5.1]{Fle89} that function $\mathrm{F}_A$ is $\mathscr{C}^\infty$ on 
$\mathbb{R}^d \times [0,T]$ and has the asymptotic expansion
\begin{equation}
\mathrm{F}_A(\mathbf{z},t) = \exp \left [ -A \cdot 
G(\mathbf{z},t) - G_0(\mathbf{z},t) - \cdots - \frac{G_{k}(\mathbf{z},t)}{A^{k}} + 
O \left ( \frac{1}{A^{k+1}} \right ) \right ]
\label{ASRd}
\end{equation}
for any $k \in \mathbb{N}$. Functions $G$ and $G_k$, $k \geqslant 0$, are locally $\mathscr{C}^\infty$ and recursively obtained by solving partial differential equations. Expansion~\eqref{ASRd} applies, in particular, to the Laplace transform of $\mathbf{Z}_A(T)$ at finite time $T$ by choosing $f(\mathbf{z}) = \langle \mathbf{u}, \mathbf{z} \rangle$ for given $\mathbf{u} \in \mathbb{R}^d$. If process $\mathbf{Z}_A$ has a stationary distribution, the Laplace transform of $\mathbf{Z}_A(\infty)$ is then deduced from~\eqref{ASRd} by letting $t \uparrow +\infty$.

Such expansions have been invoked and applied in other contexts, even if their existence is not formally stated. The analysis of coupled queuing systems is, in particular, one of the application fields of these perturbation methods (see \cite{Kne86bis, Kne95},~\cite[Chap.9]{Schuss10} and references therein). In this framework, expansions of the form
\begin{equation}
\pmb{\Pi}_A(A \mathbf{z}) = \frac{1}{(2\pi A)^{d/2}} \, 
\exp \left [ -A \cdot H(\mathbf{z}) - h_0(\mathbf{z}) - 
\frac{h_1(\mathbf{z})}{A} - \cdots \right ], 
\qquad \mathbf{z} \in \mathbb{R}^d, 
% \label{AS0}
\label{ExpAs}
\end{equation}
for the stationary distribution $\pmb{\Pi}_A$ are assumed to hold, with the decay rate $H$ and functions $h_0$, $h_1$, $\dots$ successively determined via the Kolmogorov equations. While the existence of expansion (\ref{ExpAs}) is admitted, the actual determination of unknown functions $H$, $h_0$, $h_1$, \dots\ is considered as a consistent argument for its validity. To illustrate simply the approach developed in the latter references, consider the one-dimensional processes ($d = 1$) on the half-line $[0,+\infty)$. Using (\ref{ExpAs}), the asymptotics of $\pmb{\Pi}_A(Ax)$ for fixed $x > 0$ and for small $x = n/A$, $n = O(1)$, are shown to differ by some unknown multiplying constant; this constant is determined through the “Asymptotic Matching” principle \cite[Chap.7, 7.4]{BEND99} which consists in identifying asymptotics of $\pmb{\Pi}_A(Ax)$ and $\pmb{\Pi}_A(n)$ when making $x$ tend to 0 and $n$ tend to $+\infty$, respectively. %\textbf{\textcolor{red}{Insist that the multi-dimensional case is much harder because of what happens at boundaries}}

%\commentaire{Maybe these examples deserve to be discussed in more details. Link with viscosity solutions? It would be good to justify our approach to detail at least one previous work where the expansion was assumed, as we do here}

To summarize this review, we can thus assert that 
LDP's with singular scaling are known for some specific classes of Markov processes, although not including the case of the birth-and-death process $(N,M)$ presently considered. On the other hand, the analytical approach developed in the Singular Perturbation framework can be applied for classes of processes for which no LDP is known; assuming the existence of an asymptotic expansion of the form (\ref{ExpAs}), this analytical approach then brings more precise information on the asymptotic behavior of their distribution. In this paper, admitting the existence of expansions such as (\ref{ExpAs}), the analytical approach will thus be chosen to obtain the desired asymptotics for the stationary distribution $\pmb{\Pi}_A$ of process $(N, M)$ in the regime where $A$ grows to infinity.

%%%%%%%%%%%%%%%%%%%%%%%%%%%%%%%%%%%%%%%%%%%%

\section{Preliminary results}
\label{Sec:SCH}

In this section, we introduce a scaled version of distribution $\pmb{\Pi}_A$ and explicit the Kolmogorov equation it satisfies. We also recall a Laplace expansion that will be used repeatedly.

%%%%%%%%%%%%%%%%%%%%%%%%%%%%%%%%%%%%%%%%%%%%

\subsection{Kolmogorov equations and scale change}

%%%%%%%%%%%%%%%%%%%%%%%%%%%%

Recall that $\pmb{\Pi}_A$ denotes the stationary distribution of $(N,M)$ when the stability condition~\eqref{StabCond} holds. By definition of its dynamics, it satisfies the associated set of Kolmogorov equations
\begin{align}
& \left[ \alpha + \beta + \left ( \frac{\mu \, n}{n+m} + 
\frac{\nu \, m}{n+m} \right ) \mathbf{1}_{n + m > 0} + \theta \, m \right ] \pmb{\Pi}_A(n,m) \; =
\label{Kol0} \\
& \; \alpha \, \pmb{\Pi}_A(n-1,m) \mathbf{1}_{n > 0} + 
\beta \, \pmb{\Pi}_A(n,m-1) \mathbf{1}_{m > 0} + 
\frac{\mu(n+1)}{n + m + 1} \, \pmb{\Pi}_A(n+1,m) \; + 
\nonumber \\
& (m+1) \left ( \frac{\nu }{n + m + 1} + \theta \right ) 
\pmb{\Pi}_A(n,m+1), 
\qquad \; \; (n,m) \in \mathbb{N}^2.\nonumber
\end{align}

As explained in Section~\ref{sec:model-main-results}, in the heavy traffic regime $A \to \infty$, $N$ and $M$ become of the order of $A$ and we will consequently study them on this scale. More precisely, we define the function $\mathbf{p}_A$ by
\begin{equation}
\mathbf{p}_A(x,y) = A^2 \cdot \pmb{\Pi}_A([Ax],[Ay]), \qquad x, \; y \geqslant 0,
\label{Sca0}
\end{equation}
$[x] \in \mathbb{N}$ denoting the integer part of $x \in \mathbb{R}^+$. Linear system~\eqref{Kol0} translates for 
$\mathbf{p}_A$ into the following functional equations on the open quarter-plane 
$(0,+\infty) \times (0,+\infty)$ and its boundary $\{(x,0), x \geqslant 0\} \cup \{(0,y), y \geqslant 0\}$, namely
\begin{align}
& \left[ \alpha + A \theta + \displaystyle \frac{\mu \, x}{x+y} + 
\frac{\nu \, y}{x+y} + A \theta \, y \right ] \mathbf{p}_A(x,y) \; =
\label{Kol1} \\
& \; \alpha \, \mathbf{p}_A \left (x-\frac{1}{A},y \right ) + 
A \theta \, \mathbf{p}_A \left (x,y-\frac{1}{A} \right ) 
+ \frac{\mu(Ax+1)}{A(x+y) + 1} \, \mathbf{p}_A \left (x+\frac{1}{A},y \right )
\nonumber \\
& + (Ay+1) \left ( \frac{\nu }{A(x+y) + 1} + \theta \right ) 
\mathbf{p}_A \left (x,y+\frac{1}{A} \right ), \qquad \qquad x > 0, \; \; y > 0,
\nonumber
\end{align}
in the interior quarter-plane, 
\begin{equation}
\left\{
\begin{array}{ll}
\displaystyle \left[ \alpha + A \theta + \mu \right ] \mathbf{p}_A(x,0) = 
\alpha \, \mathbf{p}_A \left (x-\frac{1}{A},0 \right ) + 
\mu \, \mathbf{p}_A \left (x+\frac{1}{A},0 \right ) 
\\
\qquad \qquad \qquad \qquad \qquad \quad \; 
+ \displaystyle \left ( \frac{\nu }{A x + 1} + \theta \right ) 
\mathbf{p}_A \left (x,\frac{1}{A} \right ), \qquad \; x > 0,
\\ \\
\displaystyle \left[ \alpha + A \theta + \nu + A \theta y \right ] 
\mathbf{p}_A(0,y) = A \theta \, \mathbf{p}_A \left (0,y-\frac{1}{A} \right ) + 
\frac{\mu}{Ay + 1} \times 
\\
\; \; \displaystyle \mathbf{p}_A \left (\frac{1}{A},y \right ) + 
(Ay+1) \left ( \frac{\nu }{A y + 1} + \theta \right ) \, 
\mathbf{p}_A \left (0,y+\frac{1}{A} \right ), \qquad y > 0,
\label{Kol1abc}
\end{array} \right.
\end{equation}
on the boundary and 
\begin{equation}
\displaystyle (\alpha + A \, \theta) \, \pmb{\Pi}_A(0,0) = 
\frac{\mu}{A^2} \, \mathbf{p}_A \left ( \frac{1}{A}, 0 \right ) + 
\frac{\nu + \theta}{A^2} \, \mathbf{p}_A \left ( 0, \frac{1}{A} \right ),
\label{Kol1c}
\end{equation}
at the origin, together with the normalization condition
\begin{equation}
\iint_{\mathbb{R}^{2+}} \mathbf{p}_A(x,y) \, \mathrm{d}x \, \mathrm{d}y = 1.
\label{Kol1bis}
\end{equation}

In the rest of the paper, we assume that the assumptions of Theorem~\ref{T2} hold, that is, $\varrho < 1$ and the expansion~\eqref{AS0} holds with functions $H$, $h_0$, $h_1$ and $h_2$ respectively of class $\mathscr{C}^3$, $\mathscr{C}^2$, $\mathscr{C}^1$ and $\mathscr{C}^0$ in the open quarter-plane $\mathbb{R}^{+*} \times \mathbb{R}^{+*}$, and function $H$ being non-negative, continuous over the closed quarter plane $\mathbb{R}^+ \times \mathbb{R}^+$ and with $H(x^*,y^*) = 0$. In terms of density function $\mathbf{p}_A$ introduced in (\ref{Sca0}), the expansion~\eqref{AS0} equivalently reads
\begin{multline}
\mathbf{p}_A(x,y) = \frac{A}{2\pi} \; \times 
\\
\exp \left [ -A \cdot H(x,y) - h_0(x,y) -
\frac{h_1(x,y)}{A} - \frac{h_2(x,y)}{A^2} + O \left ( \frac{1}{A^3} \right ) \right]
\label{Dev0}
\end{multline}
for all $x, \; y > 0$.

\begin{remark}
	An explicit solution to system~\eqref{Kol0} seems out of reach for an arbitrary set of parameters 
	$\alpha$, $\mu$, $\beta$, $\nu $ and $\theta$. To obtain an efficient approximation for this stationary distribution $\pmb{\Pi}_A$, a heuristic framework has been developed in~\cite{POAS16} on the basis of the so-called ``Quasi-Stationary approximation''. Specifically, for any state $N = n \geqslant 0$ of the number of patient customers, this approximation assumes that the conditional distribution 
	$\mathbf{D}(m \, \vert \, n) = \mathbb{P}(M(\infty) = m \, \vert \, N(\infty) = n)$, $m \in \mathbb{N}$, of $M(\infty)$, given $N(\infty) = n$, is evaluated by considering that the dynamics of process $M$ is described by keeping the value of $N$ constant in time. The Quasi-Stationary approximation proves, in particular, more robust than the direct numerical resolution of infinite system~\eqref{Kol0}. This numerical stability is beneficial, in particular, in the high load regime when $\varrho = \alpha/\mu$ tends to 1. 
	
	Remarkably, the functional $\mathfrak{S}_T$ arising in the LDP of Section~\ref{sub:LDP} involves the Quasi-Stationary distribution $\mathbf{D}(\cdot \, \vert \, \mathbf{x})$ of $\mathbf{Y}$, when fixing the state $\mathbf{x}$ of the slow process $\mathbf{X}_A$.
\end{remark}

%%%%%%%%%%%%%%%%%%%%%%%%
\subsection{Laplace expansion}
%%%%%%%%%%%%%%%%%%%%%%%%

We finally recall a classical Laplace expansion for an integral with exponential integrand and a large parameter $A$, which will be repeatedly used in the forthcoming sections. Given 

- a real (possibly infinite) interval $[a,b]$,

- real-valued functions $g$ and $h$ on $[a,b]$ such that $h \in \mathscr{C}^2[a,b]$ has a unique minimum at the interior point $r^* \in \; (a,b)$ with $h''(r^*) \neq 0$, 

- and $g \in \mathscr{C}^0[a,b]$ with $g(r^*) \neq 0$,

\noindent then~\cite[Section 5.3, Equ.(5.3.9)]{BLEI86}
\begin{equation}
\int_a^b e^{- A \cdot h(r)} g(r) \, \mathrm{d}r = e^{-A \cdot h(r^*)} 
\sqrt{\frac{2\pi}{A \, h''(r^*)}} \, g(r^*) 
\left [ 1 + O \left ( \frac{1}{A} \right ) \right ].
\label{Lapla0}
\end{equation}

Similar asymptotics hold for complex-valued integrals with the same conditions for both functions $g$ and $h$, namely 
\begin{multline}
\int_a^b e^{- A \cdot h(r) + i A \zeta r} g(r) \mathrm{d}r\\
= 
e^{-A \cdot h(r^*) + i A \zeta r^*} \cdot 
\exp \left ( \frac{- A \zeta^2}{2 h''(r^*)} \right )
\sqrt{\frac{2\pi}{A \, h''(r^*)}} \, g(r^*) 
\left [ 1 + O \left ( \frac{1}{A} \right ) \right ]
\label{Lapla0BIS}
\end{multline}
for large $A$ and any real constant $\zeta$. When either function $g$ or $h$ depends smoothly on a real parameter $\sigma$, the $O(1/A)$ remainder in~\eqref{Lapla0} or~\eqref{Lapla0BIS} tends to 0 when 
$A \uparrow +\infty$, uniformly with respect to $\sigma$ pertaining to a given compact interval.

% %%%%%%%%%%%%%%%%%%%%%%%%%%%%%%%%%%%%%%%%%
% \subsection{Functional equations}
% %%%%%%%%%%%%%%%%%%%%%%%%%%%%%%%%%%%%%%%%%
% Using scaling~\eqref{Norm0} and definition~\eqref{Sca0},

%%%%%%%%%%%%%%%%%%%%%%%%%%%%%%%%%%%%%%%%%%%%%%%%

\section{Proof of Theorem~\ref{T1}}
\label{Sec:DRF}

%%%%%%%%%%%%%%%%%%%%%%%%%%%%%%%%%%%%%%%%%%%%%%%%

Consider functional equation~\eqref{Kol1} for large $A$. Fix the point $(x,y)$ with $x > 0$ and $y > 0$; expansion~\eqref{Dev0} applied at neighboring point $(x-1/A,y)$ yields
\begin{multline*}
\mathbf{p}_A \left (x-\frac{1}{A},y \right ) =\\
\exp \left [ \, - A \cdot H\left (x-\frac{1}{A},y \right ) \, - \, h_0\left (x-\frac{1}{A},y \right ) - 
\frac{1}{A} \, h_1 \left (x-\frac{1}{A},y \right ) + \cdots \right ]
\end{multline*}
(up to factor $A/2\pi$). Using the assumed smoothness of $H$, $h_0$ and $h_1$, Taylor expansions at first order in $1/A$ near point $(x,y)$ give
\begin{multline*}
\mathbf{p}_A \left (x-\frac{1}{A},y \right ) = \; 
\exp \left [-A \, H(x,y) - h_0(x,y) - 
\frac{1}{A} \, h_1(x,y) + \cdots \right ] \times \\
\; \exp \left [ \frac{\partial H}{\partial x}(x,y) 
- \frac{1}{2A} \frac{\partial^2 H}{\partial x^2}(x,y) + 
\frac{1}{A} \frac{\partial h_0}{\partial x}(x,y) + \cdots \right ],
\end{multline*}
dots denoting $O(1/A^2)$ terms. By~\eqref{Dev0} again, the first exponential factor in the right-hand side of the latter relation equals $\mathbf{p}_A(x,y)$; expanding the second exponential term at first order in $1/A$ then gives
\begin{equation}
\frac{\mathbf{p}_A(x-1/A,y)}{\mathbf{p}_A(x,y)} = e^{+
\partial_x H} \left ( 1 - \frac{1}{A} \left [ \frac{1}{2} \, 
\frac{\partial^2 H}{\partial x^2} - \frac{\partial h_0}{\partial x} \right ] + \cdots \right ),
\label{PA1}
\end{equation}
all derivatives being taken at point $(x,y)$ ($\partial_x H$, $\partial_y H$ denote derivatives 
$\partial H/\partial x$ and $\partial H/\partial y$ for short, respectively). In a similar manner, we obtain the expansions of function $\mathbf{p}_A$ at neighboring points $(x,y-1/A)$, $(x+1/A,y)$ and $(x,y+1/A)$ in the form
\begin{equation}
\left\{
\begin{array}{ll}
\displaystyle \frac{\mathbf{p}_A(x,y-1/A)}{\mathbf{p}_A(x,y)} = 
e^{+\partial_y H} 
\left ( 1 - \frac{1}{A} \left [ \frac{1}{2} \, 
\frac{\partial^2 H}{\partial y^2} - \frac{\partial h_0}{\partial y} \right ] + 
\cdots \right ),
\\ \\
\displaystyle \frac{\mathbf{p}_A(x+1/A,y)}{\mathbf{p}_A(x,y)} = e^{ 
- \partial_x H}
\left ( 1 - \frac{1}{A} \left [ \frac{1}{2} \, 
\frac{\partial^2 H}{\partial x^2} + \frac{\partial h_0}{\partial x} \right ] + \cdots \right ),
\\ \\
\displaystyle \frac{\mathbf{p}_A(x,y+1/A)}{\mathbf{p}_A(x,y)} = e^{ 
-\partial_y H}
\left ( 1 - \frac{1}{A} \left [ \frac{1}{2} \, 
\frac{\partial^2 H}{\partial y^2} + \frac{\partial h_0}{\partial y} \right ] + \cdots \right ).
\label{PA234}
\end{array} \right.
\end{equation}
Inserting expressions~\eqref{PA1}--\eqref{PA234} into equation~\eqref{Kol1} and dividing throughout by factor $\mathbf{p}_A(x,y)$, we then obtain
\begin{align}
& \; \alpha + A \, \theta + \frac{\mu \, x}{x+y} + 
\frac{\nu \, y}{x+y} + A \, \theta y = 
\alpha \cdot e^{
\partial_x H}
\left ( 1 - \frac{1}{A} \left [ \frac{1}{2} \, 
\frac{\partial^2 H}{\partial x^2} - \frac{\partial h_0}{\partial x} \right ] + \cdots \right ) \; + 
\nonumber \\
& \; A \, \theta \cdot e^{
\partial_y H}
\left ( 1 - \frac{1}{A} \left [ \frac{1}{2} \, 
\frac{\partial^2 H}{\partial y^2} - \frac{\partial h_0}{\partial y} \right ] + \cdots \right ) \; + 
\nonumber \\
& \; \left [ \frac{\mu \, x}{x+y} + \frac{\mu \, y}{A(x+y)^2} + \cdots \right ] 
e^{- \partial_x H} 
\left ( 1 - \frac{1}{A} \left [ \frac{1}{2} \, 
\frac{\partial^2 H}{\partial x^2} + \frac{\partial h_0}{\partial x} \right ] + \cdots \right ) \; + 
\nonumber \\
& \; \left [ \frac{\nu \, y}{x+y} + \frac{\nu \, x}{A(x+y)^2} + \cdots + 
\theta(Ay+1) \right ] 
e^{- \partial_y H} 
\left ( 1 - \frac{1}{A} \left [ \frac{1}{2} \, 
\frac{\partial^2 H}{\partial y^2} + \frac{\partial h_0}{\partial y} \right ] 
+ \cdots \right ).
\nonumber
\end{align}
At order $O(A)$ and $O(1)$ for large $A$, the latter relation then entails 
\begin{equation}
\theta + \theta \, y = \theta \cdot e^{\partial_y H} + 
\theta \, y \cdot e^{- \partial_y H}
\label{EQU1}
\end{equation}
and
\begin{multline}
\alpha + \frac{\mu \, x}{x+y} + \frac{\nu \, y}{x+y} = 
\alpha \cdot e^{\partial_x H} - 
\theta \, e^{\partial_y H} \left [ \frac{1}{2} \, 
\frac{\partial^2 H}{\partial y^2} - 
\frac{\partial h_0}{\partial y} \right ]\\
+ \, \frac{\mu \, x}{x+y} \cdot e^{-\partial_x H} + 
\left [ \frac{\nu \, y}{x+y} + \theta \right ]e^{-\partial_y H} - 
\theta \, y e^{-\partial_y H} \left [ \frac{1}{2} \, 
\frac{\partial^2 H}{\partial y^2} + \frac{\partial h_0}{\partial y} \right ]
\label{EQU1BIS}
\end{multline}
respectively. We then successively observe that

\textbf{(A)} Relation~\eqref{EQU1} is a quadratic equation for 
$e^{\partial_y H}$. We can exclude the trivial solution 
$e^{\partial_y H(x,y)} = 1$ which would give $\partial_y H(x,y) = 0$ and a solution $H$ depending on variable $x$ only. We are thus left with the other solution $e^{\partial_y H(x,y)} = y$, that is, $\partial_y H(x,y) = \log y$ for $y > 0$. Integrating with respect to variable $y$, the latter relation provides
\begin{equation}
H(x,y) = \widetilde \Phi(x) + \Psi(y), \qquad x > 0, \, y > 0,
\label{HH}
\end{equation}
for a function $\widetilde \Phi$ to be determined and with function $\Psi$ given by~\eqref{eq:Psi}.

\textbf{(B)} After~\eqref{HH}, we have $\partial H/\partial y = \log y$ and 
$\partial^2 H/\partial y^2 = 1/y$ at point $(x,y)$, $x > 0$, $y > 0$. Carrying over these values into equation~\eqref{EQU1BIS}, the latter solves for the first derivative $\partial h_0/\partial y$ into
$$
\theta \, \frac{\partial h_0}{\partial y}(x,y) = \frac{1}{y-1} \left [
\alpha(1-e^{\widetilde \Phi'(x)}) + 
\frac{\mu \, x}{x+y}(1-e^{-\widetilde \Phi'(x)}) \right ] + \frac{\nu }{x+y} 
+ \frac{\theta}{2y}.
$$
Integrating the latter equality with respect to variable $y$ then yields
\begin{multline}
\theta \, h_0(x,y) = \; \theta \, \Omega(x) + 
\left [\alpha(1-e^{\widetilde \Phi'(x)}) + \frac{\mu \, x}{x+1} (1-e^{-\widetilde \Phi'(x)}) \right ] \log(y-1) \\
- \, \frac{\mu \, x}{x+1} (1-e^{-\widetilde \Phi'(x)}) \log(x+y) + \nu \log(x+y) + \frac{\theta}{2} \, \log y 
\label{h0}
\end{multline}
for all $x > 0$, $y > 0$ and some unknown function $\Omega$. By assumption, $h_0$ is continuously differentiable in the open quarter-plane and, in particular, on the vertical line $y = 1$. After relation~\eqref{h0}, this implies that the coefficient of 
$\log(y-1)$ should vanish identically, hence 
$$
\alpha(1-e^{\widetilde \Phi'(x)}) + \frac{\mu \, x}{x+1} (1-e^{-\widetilde \Phi'(x)}) = 0
$$
or, equivalently,
$$
\alpha e^{2 \, \widetilde \Phi'(x)} - \left ( \alpha + \frac{\mu \, x}{x +1} \right ) 
e^{\widetilde \Phi'(x)} + \frac{\mu \, x}{x + 1} = 0, 
\qquad x > 0.
$$
This quadratic equation for $e^{\widetilde \Phi'(x)}$ has the non-constant ($\neq 1$) solution
\begin{equation}
e^{\widetilde \Phi'(x)} = \frac{\mu x}{\alpha(x+1)} = \frac{x}{\varrho(x+1)}
\label{ExpPhi'}
\end{equation}
which differential equation readily integrates for $\widetilde \Phi$ into 
\begin{equation}
\widetilde \Phi(x) = x \log x -(x+1)\log(x+1) - x \log \varrho + C_0, \qquad x > 0,
\label{PhiC0}
\end{equation}
for some constant $C_0$. As $\Psi(y^*) = \Psi(1) = 0$, the assumption $H(x^*,y^*) = 0$ on $H$ then implies that $\widetilde \Phi(x^*) = 0$ with $x^*$ introduced in (\ref{Defx*}); this readily determines the value 
$C_0 = - \log(1-\varrho)$. The latter and~(\ref{PhiC0}) thus entirely determine the function $\widetilde \Phi$, which is thus equal to $\Phi$ defined by~\eqref{eq:Phi}. The final expression of decay rate $H = \Phi + \Psi$ in the open quarter-plane $\mathbb{R}^{+*} \times \mathbb{R}^{+*}$ follows. Since $H$ is assumed to be continuous on the closed quarter plane, this expression extends by continuity to 
$\mathbb{R}^+ \times \mathbb{R}^+$, which concludes the proof of Theorem~\ref{T1}.

\begin{remark}
Equation~\eqref{EQU1} is the so-called Hamilton-Jacobi equation for the component 
$M$~\cite[Chap.5, Theorem 4.3]{Frei12} which determines the partial derivative 
$\partial H/\partial y$ only. In the present singular Large Deviations setting, however, the full derivation of function $H$ requires another partial differential equation for the next function $h_0$, together with its smoothness across the line $y = y^* = 1$.
\end{remark}

%%%%%%%%%%%%%%%%%%%%%%%%%%%%%%%%%%%%%%%%%%%%%%%

\section{Proof of Theorem~\ref{T2}}
\label{Sec:FAE}

%%%%%%%%%%%%%%%%%%%%%%%%%%%%%%%%%%%%%%%%%%%%%%%

We now determine the prefactor $h_0$ in the expansion~\eqref{Dev0} of density $\mathbf{p}_A$. Given the expression~\eqref{PhiC0} of function $\tilde \Phi$, formula~\eqref{h0} for function $h_0$ now easily reduces to
\begin{equation}
h_0(x,y) = \Omega(x) + 
\frac{\mu}{\theta}(1-\varrho) 
\left ( \frac{x-x^*}{x+1} \right ) \log(x+y) + c \log(x+y) + \frac{\log y}{2}
\label{h0bis}
\end{equation}
for $x > 0$, $y > 0$, with $c = \nu /\theta$, $x^*$ introduced in (\ref{Defx*}) and some unknown function $\Omega$. In order to specify $\Omega$, we evaluate terms of subsequent order $O(1/A)$ in the functional equation~\eqref{Kol1} for $x > 0$ and $y > 0$. 

To this end, expansions~\eqref{PA234} for both $\mathbf{p}_A(x,y-1/A)$ and $\mathbf{p}_A(x,y+1/A)$ have to be extended up to order $O(1/A^2)$. Besides, the expansions for $\mathbf{p}_A(x \pm 1/A,y)$ at order 
$O(1/A)$ only are still sufficient. Applying then~\eqref{Dev0} at point $(x,y-1/A)$, we have 
\begin{multline*}
\mathbf{p}_A \left (x,y -\frac{1}{A} \right ) = 
\exp \Bigl [ \; - A \cdot H\left (x,y -\frac{1}{A} \right ) - 
h_0\left (x,y -\frac{1}{A} \right ) - 
\frac{1}{A} \, h_1 \left (x,y -\frac{1}{A} \right ) \\
\; - \frac{1}{A^2} \, h_2 \left (x,y -\frac{1}{A} \right ) + 
\cdots \Bigr ];
\end{multline*}
(up to multiplying factor $A/2\pi$). Writing Taylor expansions at second order in $1/A$ for functions $H$, $h_0$, $h_1$, $h_2$, \dots\ near point $(x,y)$, we then easily obtain 
\begin{align}
& \mathbf{p}_A \left (x,y -\frac{1}{A} \right ) = 
\exp \left [-A \, H(x,y) - h_0(x,y) - 
\frac{1}{A} \, h_1(x,y) - \frac{1}{A^2} \, h_2 \left (x,y \right ) 
+ \cdots \right ] \times 
\nonumber \\
& \; e^{\partial_y H} \exp \left [ 
- \frac{1}{2A} \frac{\partial^2 H}{\partial y^2} + \frac{1}{6A^2} 
\frac{\partial^3 H}{\partial y^3} + \cdots + 
\frac{1}{A} \frac{\partial h_0}{\partial x} - \frac{1}{2A^2} 
\frac{\partial^2 h_0}{\partial y^2} + \cdots + 
\frac{1}{A^2} \frac{\partial h_1}{\partial y} + \cdots \right ],
\nonumber
\end{align}
all derivatives being taken at point $(x,y)$ and dots denoting $O(1/A^3)$ terms. By expansion~\eqref{Dev0} again, the first exponential factor in the right-hand side of the latter equality equals $\mathbf{p}_A(x,y)$ (up to $A/2\pi$). Expanding the second exponential term in the right-hand side at second order in $1/A$ then provides
\begin{multline}
\frac{\mathbf{p}_A(x,y-1/A)}{\mathbf{p}_A(x,y)} = \; e^{+\partial_y H} 
\Bigl ( 1 - \frac{1}{A} \left [ \frac{1}{2} \, 
\frac{\partial^2 H}{\partial y^2} - \frac{\partial h_0}{\partial y} \right ] \;\\
+ \; \frac{1}{A^2} \left \{ \frac{1}{2} \left [ \frac{1}{2} \, 
\frac{\partial^2 H}{\partial y^2} - \frac{\partial h_0}{\partial y} \right ]^2 + \frac{1}{6} \frac{\partial^3 H}{\partial y^3} - \frac{1}{2} 
\frac{\partial^2 h_0}{\partial y^2} + \frac{\partial h_1}{\partial y} 
\right \} + \cdots \Bigr ).
\label{2PA2}
\end{multline}
At neighboring point $(x,y+1/A)$, a similar calculation yields
\begin{multline}
\frac{\mathbf{p}_A(x,y+1/A)}{\mathbf{p}_A(x,y)} = \; e^{-\partial_y H} 
\Bigl ( 1 - \frac{1}{A} \left [ \frac{1}{2} \, 
\frac{\partial^2 H}{\partial y^2} + 
\frac{\partial h_0}{\partial y} \right ] \; \\
+ \; \frac{1}{A^2} \left \{ \frac{1}{2} \left [ \frac{1}{2} \, 
\frac{\partial^2 H}{\partial y^2} + \frac{\partial h_0}{\partial y} \right ]^2 - \frac{1}{6} \frac{\partial^3 H}{\partial y^3} - \frac{1}{2} 
\frac{\partial^2 h_0}{\partial y^2} - \frac{\partial h_1}{\partial y} 
\right \} + \cdots \Bigr ).
\label{2PA4}
\end{multline}
Inserting then expansions~\eqref{2PA2},~\eqref{2PA4} and retaining terms of order 
$1/A$ in the identity following~\eqref{PA234} in the proof of Theorem~\ref{T1}, we then obtain the equation 
\begin{align}
0 = & \; -\alpha \, e^{\partial_x H} \left [ \frac{1}{2} \, 
\frac{\partial^2 H}{\partial x^2} - 
\frac{\partial h_0}{\partial x} \right ] 
\label{H0H1} \\
& \; + \theta \, e^{\partial_y H} \left \{ \frac{1}{2} \left [ \frac{1}{2} \, 
\frac{\partial^2 H}{\partial y^2} - \frac{\partial h_0}{\partial y} \right ]^2 + \frac{1}{6} \frac{\partial^3 H}{\partial y^3} - \frac{1}{2} 
\frac{\partial^2 h_0}{\partial y^2} + \frac{\partial h_1}{\partial y} 
\right \} 
\nonumber \\
& \; - \frac{\mu \, x}{x+y} \, e^{-\partial_x H} \left [ \frac{1}{2} \, 
\frac{\partial^2 H}{\partial x^2} + \frac{\partial h_0}{\partial x} \right ] 
+ \frac{\mu \, y}{(x+y)^2} \, e^{-\partial_x H} 
\nonumber \\
& \; - \left ( \frac{\nu \, y}{x+y} + \theta \right ) e^{-\partial_y H} 
\left [ \frac{1}{2} \, \frac{\partial^2 H}{\partial y^2} + 
\frac{\partial h_0}{\partial y} \right ] + 
\frac{\nu \, x}{(x+y)^2} \, e^{-\partial_y H} 
\nonumber \\
& \; + \theta y \, 
e^{-\partial_y H} \left \{ \frac{1}{2} \left [ \frac{1}{2} \, 
\frac{\partial^2 H}{\partial y^2} + \frac{\partial h_0}{\partial y} \right ]^2 - \frac{1}{6} \frac{\partial^3 H}{\partial y^3} - \frac{1}{2} 
\frac{\partial^2 h_0}{\partial y^2} - \frac{\partial h_1}{\partial y} 
\right \}
\nonumber
\end{align}
involving $\partial h_0/\partial x$ and $\partial h_1/\partial y$. The derivative 
$\partial h_0/\partial x$ intervenes in~\eqref{H0H1} in the first and third brackets only, with multiplying coefficient
\begin{equation}
K(x,y) = \alpha \, e^{\partial_x H} - 
\frac{\mu \, x}{x+y} \, e^{-\partial_x H} = \mu \, \frac{x(x+y) - \varrho(x+1)^2}{(x+1)(x+y)},
\label{CoeffK}
\end{equation}
after using expression~\eqref{ExpPhi'} for $\widetilde \Phi'(x) = \partial H(x,y)/\partial x$. Calculating the derivatives 
$\partial H/\partial y = \log y$, $\partial^2 H/\partial y^2 = 1/y$, 
$\partial^3 H/\partial y^3 = -1/y^2$ together with 
$$
\frac{\partial h_0}{\partial y}(x,y) = \frac{c}{x+y} + \frac{1}{2y}, \quad 
\frac{\partial^2 h_0}{\partial y^2}(x,y) = - \frac{c}{(x+y)^2} - \frac{1}{2y^2}
$$
after~\eqref{h0bis}, equation~\eqref{H0H1} then reads
\begin{equation}
K(x,y) \frac{\partial h_0}{\partial x}(x,y) - L(x,y) = \theta(1-y) 
\frac{\partial h_1}{\partial y}(x,y), \qquad x > 0, \; y > 0,
\label{H0H1bis}
\end{equation}
when isolating each derivative $\partial h_0/\partial x$, 
$\partial h_1/\partial y$ and setting
\begin{multline}
L(x,y) = \; \frac{\mu}{2(x+1)^2} + \frac{\alpha}{2x(x+y)} - 
\frac{\alpha(x+1)y}{x(x+y)^2} - \frac{\nu \, x}{(x+y)^2 y} - 
\frac{c(1+c)\theta y}{2(x+y)^2}\\
- \frac{\theta}{12 y} - \theta \left [ \frac{c(1+c)}{2(x+y)^2} + \frac{c}{(x+y)y} + \frac{11}{12 \, y^2} \right ] + \frac{1}{y} \left ( \frac{\nu \, y}{x + y} + \theta \right ) 
\left ( \frac{1}{y} + \frac{c}{x + y} \right ).
\label{CoeffL}
\end{multline}
By assumption, $h_1$ is of class 
$\mathscr{C}^1$ in the open quarter-plane and, in particular, on the vertical line 
$y = y^* = 1$. In view of functional relation~\eqref{H0H1bis}, this implies that its left-hand side should identically vanish for $y = y^* = 1$, that is,
\begin{equation}
\forall \; x > 0, \qquad \frac{\partial h_0}{\partial x}(x,1) = 
\frac{L(x,1)}{K(x,1)}.
\label{H0H1ter}
\end{equation}
By expressions~\eqref{CoeffK} and~\eqref{CoeffL} of $K(x,y)$ and 
$L(x,y)$, elementary algebra provides
$$
K(x,1) = \mu(1-\varrho) \, \frac{x-x^*}{x + 1}, \qquad
L(x,1) = \mu(1-\varrho) \, \frac{x-x^*}{2x(x + 1)^2}
$$
(note that both rational fractions $K(x,1)$ and $L(x,1)$ have a simple zero at $x = x^*$ so that the ratio $L(x,1)/K(x,1)$ is well-defined for all $x > 0$). Using the latter,~\eqref{H0H1ter} then gives 
$\partial h_0(x,1)/\partial x = 1/\left[2x(x+1)\right]$, $x > 0$ which readily integrates to
\begin{equation}
h_0(x,1) = C_0 + \log \sqrt{\frac{x}{x+1}}, \qquad x > 0,
\label{Resolh0}
\end{equation}
for some constant $C_0$. Besides, expression~\eqref{h0bis} for $h_0(x,y)$ readily shows that the difference $h_0(x,y) - h_0(x,1)$ is independent of the function $\Omega$ and equals
\begin{equation}
h_0(x,y) - h_0(x,1) = \frac{\log y}{2} + \left [ c + 
\frac{\mu}{\theta}(1-\varrho) \, \left ( \frac{x-x^*}{x+1} \right ) 
\right ] \log \left ( \frac{x+y}{x+1}\right ).
\label{h0h01}
\end{equation}
Using relation~\eqref{h0h01}, we thus deduce that $h_0(x,y) = h_0(x,1) + (h_0(x,y) - h_0(x,1))$ eventually equals
$$
h_0(x,y) = C_0 + \log \sqrt{\frac{x}{x+1}} + \frac{\log y}{2} + \left [ c + 
\frac{\mu}{\theta}(1-\varrho) \, \left ( \frac{x-x^*}{x+1} \right ) 
\right ] \log \left ( \frac{x+y}{x+1}\right )
$$
for $x > 0$, $y > 0$. At first order in $1/A$, the expansion~\eqref{Dev0} for density $\mathbf{p}_A$ in the interior quarter plane therefore reads
\begin{multline}
\mathbf{p}_A(x,y) \sim \; \frac{A \, e^{-C_0}}{2\pi \sqrt{y}} \, 
e^{-A \cdot H(x,y)} \, \sqrt{\frac{x+1}{x}} \; \\
\times \; \exp \left [ \left \{ c + \frac{\mu}{\theta}(1-\varrho) 
\left ( \frac{x-x^*}{x+1} \right ) \right \} 
\log \left ( \frac{x+1}{x+y} \right ) \right ], \qquad x > 0, \; y > 0,
\label{AsymptPA0}
\end{multline}
which determines $\mathbf{p}_A$ in the interior quarter plane, up to constant $e^{-C_0}$. The latter is determined by condition~\eqref{Kol1bis}, once written as 
$\int_{\mathbb{R}^{+*} \times \mathbb{R}^{+*}} \mathbf{p}_A(x,y) 
\mathrm{d}x \, \mathrm{d}y \sim 1$; using~\eqref{AsymptPA0} and applying asymptotics~\eqref{Lapla0} successively to the integral with respect to variable $x$ and to variable $y$ in the latter, we obtain 
$e^{-C_0} = 1 - \varrho$. After the definition~\eqref{Sca0} of $\pmb{\Pi}_A$ in terms of $\mathbf{p}_A$, expression~\eqref{AsymptPA} eventually follows. This concludes the proof of Theorem~\ref{T2}, from which we can deduce the following corollary.

\begin{corol}
Given the assumptions of Theorem \ref{T2}, the marginal stationary distributions $N(\infty)$ and $M(\infty)$ are respectively asymptotic to
\begin{equation}
\left\{
\begin{array}{ll}
\mathbb{P}(N(\infty) = Ax) \sim \displaystyle \frac{1-\varrho}{\sqrt{2\pi A \,}} \, 
\sqrt{\frac{x+1}{x}} \; e^{-A \cdot \Phi(x)}, \qquad x > 0,
\\ \\
\mathbb{P}(M(\infty) = Ay) \sim \displaystyle \frac{1}{\sqrt{2\pi A \, y}} \, 
\frac{e^{-A \cdot \Psi(y)}}{(1-\varrho)^c(x^*+y)^c}, \qquad y > 0,
\end{array} \right.
\label{DensqrA}
\end{equation}
for large $A$.
\label{C1}
\end{corol}

\begin{proof}
	% As a consequence of Theorem~\ref{T2}, we can readily derive the asymptotics~\eqref{DensqrA} of Corollary~\ref{C1} for the marginal distributions of components $N_A$ and $M_A$. In fact,
	Integrating expression~\eqref{AsymptPA} with respect to variable $y > 0$ and applying the Laplace expansion~\eqref{Lapla0} at the unique minimum of function $\Psi$ at point $y = y^*$, asymptotics~\eqref{DensqrA} for $N(\infty)$ follows. Integrating in turn~\eqref{AsymptPA} with respect to variable $x > 0$ and applying Laplace expansion~\eqref{Lapla0} at the unique minimum of function $\Phi$ at point $x = x^*$, asymptotics~\eqref{DensqrA} for $M(\infty)$ is similarly derived from Theorem~\ref{T2}.
\end{proof}

\begin{remark}
In a way similar to that used in this Section, sharp asymptotics of density 
$\mathbf{p}_A$ on the boundary $\{(x,0), x \geqslant 0\} \cup \{(0,y), 
y \geqslant 0\}$ could be derived from equations~\eqref{Kol1abc}--\eqref{Kol1c}. Such evaluations are not needed in the present study and we only sketch the resolution procedure. For $x > 0$ and $y > 0$, asymptotic matching arguments can be first invoked to set
$$
\mathbf{p}_A(x,0) \sim \frac{A^{\frac{3}{2}}}{2\pi} \cdot \varphi(x) \, 
e^{-A(\Phi(x)+\Psi(0))}, \quad
\mathbf{p}_A(0,y) \sim \frac{A^{\frac{3}{2}}}{2\pi} \cdot \psi(x) \, 
e^{-A(\Phi(0)+\Psi(y))}
$$
for some functions $\varphi$, $\psi$, together with
$$
\mathbf{p}_A\left( x,\frac{1}{A} \right) \sim \frac{A^{\frac{5}{2}}}{2\pi} 
\cdot \varphi_1(x) \, e^{-A(\Phi(x)+\Psi(0))}, \quad
\mathbf{p}_A\left( \frac{1}{A},y \right) \sim \frac{A^{\frac{5}{2}}}{2\pi} 
\cdot \psi_1(y) \, e^{-A(\Phi(0)+\Psi(y))}
$$
where $\varphi_1$, $\psi_1$ can be derived from~\eqref{AsymptPA}. Each equation~\eqref{Kol1abc} then provides the respective solution for $\varphi$ and 
$\psi$ by identifying $O(1)$ terms for large $A$. The last equation~\eqref{Kol1c} gives the final asymptotics for $\pmb{\Pi}_A(0,0)$.
\end{remark}

%%%%%%%%%%%%%%%%%%%%%%%%%%%%%%%%%%%%%%%%%%%%%%%

\section{Proof of Theorem~\ref{T0}}
\label{Sec:EOB}

%%%%%%%%%%%%%%%%%%%%%%%%%%%%%%%%%%%%%%%%%%%%%%%

Define the generating function $\mathrm{F}_A$ of the pair $(N(\infty),M(\infty))$ by 
\begin{equation}
\mathrm{F}_A(\UU,\VV) = \mathbb{E} \left( \UU^{N(\infty)}\VV^{M(\infty)} \right), \qquad 
(\UU,\VV) \in \mathbb{D} \times \mathbb{D},
\label{DefFA}
\end{equation}
where $\mathbb{D}$ is the open unit disk. The sharp asymptotics for $\pmb{\Pi}_A$ stated in Theorem~\ref{T2} in the interior quarter-plane $\mathbb{R}^{+*} \times \mathbb{R}^{+*}$ are now applied to obtain estimates for generating function $\mathrm{F}_A$ in a relevant domain. As a preamble, we first show that $\mathrm{F}_A$ has an analytic continuation from the product $\mathbb{D} \times \mathbb{D}$ to a larger domain containing a neighborhood of point $(\UU,\VV) = (1,1)$.

\begin{lemma} 
Given $\varrho < 1$, the generating function $\mathrm{F}_A$ can be analytically extended to the product domain
\begin{equation}
\pmb{\Omega} = \mathbb{D} \left(0,\frac{1}{\varrho} \right) \times \mathbb{C}
\label{DefOmega}
\end{equation}
where $\mathbb{D}(0,\frac{1}{\varrho})$ is the open disk centered at $\UU = 0$ and with radius $1/\varrho$.
\label{L0}
\end{lemma}

\noindent
The proof is detailed in Appendix~\ref{PL0}. The main steps sum up as follows: a sample path property of process $M$ first ensures the existence of $F_A(\UU,\VV)$ for all $(\UU,\VV) \in \mathbb{D} \times \mathbb{C}$; an estimate of the marginal distribution of $N(\infty)$ justifies in turn the existence for 
$(\UU,\VV) \in \mathbb{D} \left ( 0,1/\varrho \right) \times \mathbb{D}$; finally, a convexity property of the convergence domain of power series $\mathrm{F}_A(\UU,\VV)$ concludes for its finiteness over 
$\pmb{\Omega}$. 

Now, consider the open subset $\pmb{\Omega}' \subset \pmb{\Omega}$ defined by
$$
\pmb{\Omega}' = \pmb{\Omega} \setminus 
\left\{u, \; u \in \left ( - 1/\varrho,0\right] \right\} \times 
\{v, \; v \in \; (-\infty,0]\}
$$
(we have thus excluded the non positive real points $(\UU,\VV)$ from $\pmb{\Omega}$). Using Theorems~\ref{T1} and~\ref{T2}, we can then assert the following.

\begin{prop}
Given $\varrho < 1$ and the assumptions of Theorem \ref{T2}, the generating function 
$\mathrm{F}_A$ of the pair $(N(\infty), M(\infty))$ is asymptotic for large $A$ to
\begin{multline}
\mathrm{F}_A(\UU,\VV) \sim \, \left ( \frac{1-\varrho}{1-\varrho r}\right )^A 
\, e^{A(s-1)} \, \exp \left [ iA \left ( \frac{\varrho r \zeta}{1-\varrho r} + 
s \eta \right ) \right ] \;\\
\times \, \exp \left [ - \frac{A}{2} \left ( \frac{\varrho r \zeta^2}{(1 - \varrho r)^2} + 
s \eta^2 \right ) \right ] \, G_0(\UU,\VV)
\label{HEDP2}
\end{multline}
for $(\UU,\VV) \in \pmb{\Omega}'$, where we set $\UU = r \, e^{i\zeta}$, $0 < r < 1/\varrho$, 
$\zeta \in \; (-\pi,\pi)$, and $v = s \, e^{i\eta}$, $s > 0$, 
$\eta \in \; (-\pi,\pi)$, respectively and where the continuous function $G_0$ is given by
$$
G_0(\UU,\VV) = \left ( \frac{1-\varrho}{1-\varrho r}\right ) 
\bigl [ s + \varrho r (1-s) \bigr ]^{\frac{\alpha}{\theta}(1-r)-c} 
$$
with $G_0(1,1) = 1$.
\label{P3}
\end{prop}

\noindent
The proof is detailed in Appendix~\ref{PL1}.

\begin{remark}
After the general result~\eqref{Lapla0BIS}, asymptotics~\eqref{HEDP2} can be also specified by stating a remainder term of order $O(1/A)$ for large $A$ which tends to 0, uniformly with respect to variable 
$(\UU,\VV)$ pertaining to any compact subset of domain~$\pmb{\Omega}'$.
\label{Runif}
\end{remark}

We can now proceed with the proof of Theorem~\ref{T0}. 

\begin{proof} [Proof of Theorem~\ref{T0}] 
First address the weak convergence of the pair $(\xi_A,\eta_A)$. Let $\mathrm{L}_A$ be the characteristic function of random variable $(\xi_A,\eta_A)$; we have
\begin{align}
	\mathrm{L}_A(\sigma,\tau) = & \, 
	\mathbb{E} \left ( \exp \left [ i\sigma \sqrt{A} 
	\left ( \frac{N(\infty)}{A} - x^* \right ) + i \tau \sqrt{A} 
	\left ( \frac{M(\infty)}{A} - 1 \right ) \right ] \right )
	\label{CV0} \\ 
	= & \, e^{-i\sqrt{A}(\sigma \, x^*+ \tau)} \, 
	\mathrm{F}_A ( e^{\frac{i\sigma}{\sqrt{A}}}, e^{\frac{i \tau}{\sqrt{A}}} )
	\nonumber
	\end{align}
for all $(\sigma,\tau) \in \mathbb{R}^2$, with $i^2 = -1$. Apply then Proposition~\ref{P3} to the point $(\UU,\VV)$ with $\UU = \exp(i\sigma/\sqrt{A})$ and $\VV = \exp(i\tau/\sqrt{A})$, pertaining to a neighborhood of $(1,1)$ for large enough $A$. We clearly have $r = \vert \UU \vert = 1$ and 
	$s = \vert \VV \vert = 1$, so that $G_0(\UU,\VV) = 1$ and asymptotics~\eqref{HEDP2} presently reduces to 
	\begin{multline*}
	\mathrm{F}_A ( e^{\frac{i\sigma}{\sqrt{A}}}, e^{\frac{i \tau}{\sqrt{A}}} ) =\\
	\exp \left [ iA \left ( x^* \sigma + \tau \right ) \frac{1}{\sqrt{A}} \right ] 
	\cdot 
	\exp \left [ - \frac{A}{2} 
	\left ( \frac{\sigma^2}{A} \, \frac{\varrho}{(1-\varrho)^2} + \frac{\tau^2}{A} 
	\right ) \right ] \; \times
	 \, \left [ 1 + O \left ( \frac{1}{A} \right ) \right ]
	\end{multline*}
	where, after Remark~\ref{Runif}, the remainder term $O(1/A)$ tends to 0 uniformly in variables 
	$(\sigma,\tau)$ (in fact, the pair $(\UU,\VV) = ( e^{i\sigma/\sqrt{A}}, e^{i \tau/\sqrt{A}} )$ pertains to a compact neighborhood of point $(1,1) \in \pmb{\Omega}'$ for large enough $A$). It then follows that 
	\begin{equation}
	\mathrm{F}_A ( e^{\frac{i\sigma}{\sqrt{A}}}, e^{\frac{i \tau}{\sqrt{A}}} ) \sim 
	e^{+i\sqrt{A}(\sigma \, x^*+ \tau)} \; 
	\exp \left [ - \frac{1}{2} 
	\left ( \frac{\varrho \sigma^2}{(1-\varrho)^2} + \tau^2 \right ) \right ]
	\label{CV1}
	\end{equation}
	for large $A$ and any given $(\sigma,\tau) \in \mathbb{R}^2$. By equality~\eqref{CV0} and estimate~\eqref{CV1}, we thus derive that
	$$
	\lim_{A \uparrow +\infty} \mathrm{L}_A(\sigma,\tau) = 
	\exp \left [ - \frac{1}{2} 
	\left ( \frac{\varrho \sigma^2}{(1-\varrho)^2} + \tau^2 \right ) \right ], 
	\qquad (\sigma,\tau) \in \mathbb{R},
	$$
	which limit defines the characteristic function of the Gaussian distribution with covariance matrix given in Theorem~\ref{T0}. By L\'evy's continuity Theorem~\cite[Chap.19, Theorem 19.1]{Jacod04}, we conclude that the scaled random variable $(\xi_A,\eta_A)$ converges weakly towards this Gaussian distribution.
	
Finally consider the estimation of expectations $\mathbb{E}(N(\infty))$ and $\mathbb{E}(M(\infty))$. Note that, in general, the latter weak convergence of $(\xi_A,\eta_A)$ does not necessarily imply that $\mathbb{E}(\xi_A) \to \mathbb{E}(\xi)$ and $\mathbb{E}(\eta_A) \to\mathbb{E}(\eta)$. Presently, however, we can directly rely on the asymptotics of Corollary (\ref{C1}) for the marginal distributions of $N(\infty)$ to write
$$
\mathbb{E}(N(\infty)) = \sum_{n \geqslant 0} n \, \mathbb{P}(N(\infty) = n) \sim 
\int_0^{+\infty} (A \, x) \, \mathbb{P}(N(\infty) = A \, x) \, A \, \mathrm{d}x
$$
for large $A$, after estimating the discrete sum by a Riemann integral with integral step $1/A$; using asymptotics (\ref{DensqrA}) for $\mathbb{P}(N(\infty) = Ax)$, the latter then entails
$$
\mathbb{E}(N(\infty)) \sim 
A^2 \left ( \frac{1-\varrho}{\sqrt{2\pi A}}\right )
\int_0^{+\infty} \sqrt{x(x+1)}\, e^{-A \cdot \Phi(x)} \, \mathrm{d}x.
$$
Applying the Laplace asymptotics (\ref{Lapla0}) to the latter integral, with the minimum of $\Phi$ located at $x = x^*$ with $\Phi(x^*) = \Phi'(x^*) = 0$ and $\Phi''(x^*) = (1-\varrho)^2/\varrho$, we readily obtain $\mathbb{E}(N(\infty)) \sim A \, x^*$ as claimed. A similar calculation provides $\mathbb{E}(M(\infty)) \sim A \, y^*$ for large $A$.
\end{proof}

\section{Conclusion}

%%%%%%%%%%%%%%%%%%%%%%%%%%%%%%%%%%%%%

In this paper, sharp large deviations asymptotics and limit theorems for the stationary queue occupancy distribution have been derived for the Processor-Sharing queue with both patient and impatient customers, in the case when the normalized arrival rate $A$ of impatient customers grows to infinity. On mathematical ground, the asymptotic setting is a new case of \textit{singular pertubation} for the underlying bi-dimensional birth-and-death process where the time scale of one component is accelerated while that of the other component is kept fixed. As no general large deviations principle is available for such a Markov process with discrete state space, the sharp asymptotics have been obtained by assuming an expansion of the form 
$$
\mathbf{p}_A = \frac{A}{2\pi} \, e^{-A \cdot H} \, \left ( g + \frac{g_1}{A} + \cdots \right ), 
\qquad A \uparrow +\infty, 
$$
for the scaled solution $\mathbf{p}_A$ to Kolmogorov equations. We have shown how unknown functions $H$, $g$, \dots\ can be iteratively determined. 

These results have been applied to the \textit{closed-loop} PS queue fed back by the flow of impatient customers with still uncompleted service. Unlike the common queueing systems with growth 
$1/(1-\varrho_\tot)^\alpha$ in high load condition for some $\alpha > 0$, this closed-loop PS queue has been shown to exhibit a slower logarithmic growth $-\log(1-\varrho_\tot)$ in the high load regime. In performance terms, the account of the so-called moving users is beneficial to the system behavior and the throughput of each user class decays less fast in case of congestion, as per estimates~\eqref{Throughps}.

The present approach offers generalizations when extended to queuing systems with a state space with higher dimension. Specifically, consider the PS queue with a number $K$ of patient or impatient customer classes, with arrival rate $\alpha_k$, service rate $\mu_k$ and impatient rate $\theta_k \geqslant 0$ for class $k \in \{1,\ldots,K\}$. This system should be amenable to the techniques applied in the present paper when the arrival rate $\alpha_k$, with $\theta_k \neq 0$, of some class $k$ of impatient customers tends to infinity proportionally to a dimensionless parameter $A$. While the present approach has directly considered asymptotics for the solution of the Kolmogorov equations in dimension $K = 2$, an alternative approach for $K > 2$ consists in deriving asymptotics for the generating function $\mathrm{F}_A$ of the queue occupancy $(N_1,\ldots,N_K)$. In fact, it can be easily shown from system~\eqref{Kol0} that $\mathrm{F}_A$ verifies the integro-differential equation
\begin{multline*}
\left [ \sum_{k=1}^K \alpha_k(1-\UU_k)\right ] 
\mathrm{F}_A(\mathbf{u}) + \sum_{k=1}^K \theta_k \, (\UU_k-1) \frac{\partial \mathrm{F}_A}{\partial \UU_k}(\mathbf{u}) \; = \\
\int_0^1 \left [ \sum_{k=1}^K \mu_k(1-\UU_k) 
\frac{\partial \mathrm{F}_A}{\partial \UU_k}(t\mathbf{u}) \right ] \mathrm{d}t, 
\qquad \mathbf{u} = (\UU_1,\ldots,\UU_K) \in \mathbb{D}^K, 
\end{multline*}
with $\mathrm{F}_A(1,\ldots,1) = 1$. In some extended analyticity domain $\pmb{\Omega} \supset 
\mathbb{D}^K$, an expansion
$$
\mathrm{F}_A = e^{-A \cdot G} \, \left ( G_0 + \frac{G_1}{A} + \cdots \right )
$$
for $\mathrm{F}_A$ could then be determined through the latter equation and provide general information on the corresponding multivariate queue distribution.

%%%%%%%%%%%%%%%%%%%%%%%
\appendix
%%%%%%%%%%%%%%%%%%%%%%%

%%%%%%%%%%%%%%%%%%%%%%%%%%%%%%%%%%%%%%%%%%%%
\section{Derivation of decay rate $K$}
\label{PL1bis}
%%%%%%%%%%%%%%%%%%%%%%%%%%%%%%%%%%%%%%%%%%%%
In this appendix, we prove that the component $\Phi$ of $H$ is the decay rate related to the single-server PS queue with $A$ permanent customers, arrival rate $\alpha$ and service rate $\mu$, as was claimed in Section~\ref{sec:model-main-results}. More generally, assume $\varrho = \alpha/\mu < 1$ and let $\mathbf{E}_m$ denote the stationary distribution of the single-server PS queue with a fixed number $m$ of permanent customers in queue, arrival rate $\alpha$ and service rate $\mu$.

\begin{lemma} \label{lemma:K}
	Consider $x = O(1)$ and $y = O(1)$. We then have
	\begin{equation}
	\lim_{A \uparrow +\infty} \frac{1}{A} \cdot \log \mathbf{E}_{Ay}(Ax) = - K(x,y)
	\label{PS_GenAsympt}
	\end{equation}
where
	$$
	K(x,y) = x \log \left ( \frac{x}{\varrho} \right ) + y \log y - 
	(x+y)\log(x+y) - y \log(1-\varrho).
	$$
\end{lemma}

Since $K(x,1) = \Phi(x)$ for 
$y = 1$, this indeed shows that $\Phi$ is the decay rate of the single-server PS queue with $A$ permanent customers.

\begin{proof} [Proof of Lemma~\ref{lemma:K}] By a simple reversibility argument for the Markov chain representing the queue occupancy, we first have
	%
	% In fact, consider the \textbf{M/M/1 Processor-Sharing} queue with arrival rate
	% $\alpha$, service rate $\mu$ and $m$ permanent customers; assuming
	% $\varrho = \alpha/\mu < 1$, its stationary queue distribution is given by
	\begin{equation}
	\mathbf{E}_m(n) = \varrho^n 
	\prod_{k=1}^n \left ( 1 + \frac{m}{k} \right ) \times \mathbf{E}_m(0), 
	\qquad n \in \mathbb{N},
	\label{PS_Gen}
	\end{equation}
with $\mathbf{E}_m(0)$ given by the normalization condition. More precisely, $\mathbf{E}_m(0) = 1/R_m(\varrho)$ where 
	\begin{equation}
	R_m(z) = \sum_{n \geqslant 0} \frac{z^n}{n!} \prod_{k=1}^n (k + m) = 
	\frac{1}{(1-z)^{m+1}}, \qquad 0 < z < 1,
	\label{Rm}
	\end{equation}
hence $\mathbf{E}_m(0) = (1-\varrho)^{m+1}$. Now address the estimation of $\mathbf{E}_{Ay}(Ax)$ for large $A$ and fixed $x > 0$, $y > 0$. The logarithm of the product 
	$$
	\mathbf{W}_{Ay}(Ax) = \prod_{1 \leqslant k \leqslant Ax} 
	\left ( 1 + \frac{Ay}{k} \right )
	$$
involved in expression~\eqref{PS_Gen} where $n = Ax$ and $m = Ay$, with $x = O(1)$ and $y = O(1)$, can be written as the sum
	\begin{equation}
	\log \mathbf{W}_{Ay}(Ax) = T(Ay) + Ay \sum_{k=1}^{Ax} \frac{1}{k} - 
	\sum_{k = Ax+1}^{+\infty} \mathbf{g}_{Ay}(k)
	\label{XiA}
	\end{equation}
where we set
	$$
	\mathbf{g}_m(u) = \log \left ( 1 + \frac{m}{u} \right ) - \frac{m}{u}, \qquad 
	T(z) = \sum_{j \geqslant 1} \mathbf{g}_{z}(j)
	$$
for $u \in [1,+\infty)$ and 
$z > 1$, respectively. We successively evaluate each term of the right-hand side of~\eqref{XiA} for large $A$: 

\textbf{a)} by the Weierstrass product formula~\cite[Sect. 5.8.2]{NIST10}, the sum $T(z)$ can be first made explicit in terms of the $\Gamma$ function only, namely
	$$
	T(z) = - \log \Gamma(z) - \gamma z - \log z, \qquad z > 0,
	$$
$\gamma$ denoting the Euler constant. Using this expression of $T(z)$ and the Stirling's asymptotic formula $\log \Gamma (z) = z \log z - z - (\log z)/2 + \log\sqrt{2\pi} + o(1)$ for large positive $z$ \cite[Sect. 5.11.1]{NIST10}, we thus obtain
	\begin{equation}
	T(Ay) = - Ay \cdot \log(Ay) + (1-\gamma)Ay - \frac{1}{2} \log(Ay) - 
	\log\sqrt{2\pi} + o(1);
	\label{T_A}
	\end{equation}

\textbf{b)} besides, the second term in the right-hand side of~\eqref{XiA} is proportional to the harmonic sum, which is known to expand as~\cite[Sect. 2.10.8]{NIST10}
	\begin{equation}
	\sum_{j=1}^{Ax} \frac{1}{j} = \log (Ax) + \gamma + \frac{1}{2Ax} + o(1);
	\label{Harm}
	\end{equation}

\textbf{c)} finally, the last sum in the right-hand side of~\eqref{XiA} can be written via the Euler-MacLaurin formula~\cite[Sect. 2.10.1]{NIST10} in the form
	\begin{multline}
	\sum_{k = Ax + 1}^{+\infty}
	\mathbf{g}_{Ay}(k) = \, 
	\int_{Ax+1}^{+\infty} \mathbf{g}_{Ay}(u) \mathrm{d}u + 
	\frac{\mathbf{g}_{Ay}(+\infty) + \mathbf{g}_{Ay}(Ax+1)}{2}\\
	+ \, \frac{1}{12} \, (\mathbf{g}'_{Ay}(+\infty) - \mathbf{g}'_{Ay}(Ax+1)) + \cdots
	\label{EMLbis0}
	\end{multline}
From the derivative $\mathbf{g}'_m(u) = m^2/u^2(u+m)$, $u \geqslant 1$, we have $\mathbf{g}'_{Ay}(+\infty) = 0$ and $\mathbf{g}'_{Ay}(Ax+1) = o(1)$. Besides, calculating the integral in the right-hand side of~\eqref{EMLbis0} gives
	\begin{align}
	\int_{Ax+1}^{+\infty} \mathbf{g}_{Ay}(u) \mathrm{d}u & \, = 
	\Bigl[ u \, \mathbf{g}_{Ay}(u) \Bigr]_{u=Ax+1}^{+\infty} - 
	\int_{Ax+1}^{+\infty} u \, \mathbf{g}'_{Ay}(u) \mathrm{d}u
	\label{EMLbis1} \\
	& \, = -(Ax+1)\mathbf{g}_{Ay}(Ax+1) - Ay \cdot 
	\log \left ( \frac{Ax+Ay+1}{Ax+1} \right )
	\nonumber
	\end{align}
by using an integration by parts along with the previous expression of $\mathbf{g}'_m(u)$; furthermore, the factor $\mathbf{g}_{Ay}(Ax+1)$ in the right-hand side of~\eqref{EMLbis1} expands as
	$$
	\mathbf{g}_{Ay}(Ax+1) = \log \left ( 1 + \frac{y}{x} \right ) - \frac{y}{x} + 
	\frac{y^2}{x^2 (x+y) A} + o \left ( \frac{1}{A} \right );
	$$
gathering expression~\eqref{EMLbis1} and the former results, the sum~\eqref{EMLbis0} can consequently be evaluated as
	\begin{multline} \label{EMLbis2}
	\sum_{k = Ax + 1}^{+\infty} \mathbf{g}_{Ay}(k) = \; 
	\Bigl \{ -A \left[ x \log \left ( 1 + \frac{y}{x} \right ) - y \right] 
	- \left [ \log \left ( 1 + \frac{y}{x} \right ) - 
	\frac{y}{x} \right ] \Bigr \}\\
	- \frac{y^2}{x (x+y)} 
	- \Bigl \{ A y \log \left ( 1 + \frac{y}{x} \right ) - \frac{y^2}{x(x+y)} \Bigr \} + \frac{1}{2} \left [ \log \left ( 1 + \frac{y}{x} \right ) - \frac{y}{x} 
	\right ] + o(1)
	\end{multline}
after expanding all contributing terms up to order $1/A$. After~\eqref{T_A},~\eqref{Harm} and~\eqref{EMLbis2}, we conclude that the logarithm~\eqref{XiA} expands as 
	\begin{multline}
	\log \mathbf{W}_{Ay}(Ax) = \; -A(y\log y - (x+y)\log(x+y) + x \log x)\\
	- \log(\sqrt{2\pi A y}) \; + \frac{1}{2} \log \left ( 1 + \frac{y}{x} \right ) + o(1).
	\label{WW}
	\end{multline}

Coming back to the expression~\eqref{PS_Gen} of probability $\mathbf{E}_{Ay}(Ax)$, and using the value of $\mathbf{E}_{Ay}(0)$ obtained after~\eqref{Rm} yields
	\begin{align}
	\mathbf{E}_{Ay}(Ax) = & \; \varrho^{Ax} \cdot 
	 \mathbf{W}_{Ay}(Ax) \cdot \mathbf{E}_{Ay}(0) 
	\label{WWbis} \\
	= & \; (1-\varrho) \, 
	\exp \left [ A \, x \log \varrho + \log \mathbf{W}_{Ay}(Ax) + A \, 
	y \log(1-\varrho) \right ];
	\nonumber
	\end{align}
inserting the expansion~\eqref{WW} for $\log \mathbf{W}_{Ay}(Ax)$ into equality (\ref{WWbis}) then provides the limit~\eqref{PS_GenAsympt} with the expected decay rate $K(x,y)$. 
\end{proof}

\begin{remark}
Note for completeness that a sharp asymptotics for $\mathbf{E}_{Ay}(Ax)$ also readily follows from~\eqref{WW}--\eqref{WWbis}, giving
\begin{equation}
	\mathbf{E}_{Ay}(Ax) \sim \frac{1-\varrho}{\sqrt{2\pi A}} 
	\sqrt{\frac{x+y}{x \, y}} \, \exp(-A \cdot K(x,y))
\label{Exy}
\end{equation}
for large $A$. For any real $r = O(1)$, in particular, write 
$x = x^* + r/\sqrt{A}$ and 
$y = 1$; a Taylor expansion then gives $K(x,1) = \Phi(x) = \Phi''(x^*) r^2/2A + o(1/A)$ 
so that, after asymptotics (\ref{Exy}), 
\begin{align}
\mathbf{E}_{A}(Ax^*+ r\sqrt{A}) \sim & \, \frac{1-\varrho}{\sqrt{2\pi A}} \sqrt{\frac{x^*+1}{x^* \cdot 1}} \times e^{-A \cdot K(x,y)}
\nonumber \\
\sim & \, \frac{1-\varrho}{\sqrt{\varrho}} \, \frac{1}{\sqrt{2\pi A}} \cdot \exp \left [ - \, \frac{(1-\varrho)^2}{\varrho} \, \frac{r^2}{2} \right ], \qquad r \in \mathbb{R}.
\label{Exybis}
\end{align}
We thus conclude from (\ref{Exybis}) that the probability $\mathbb{P}(N'(\infty) = Ax^* + r\sqrt{A})$ is asymptotic to $\mathbb{P}(\xi = r)/\sqrt{A}$, where $\mathbb{P}(\xi = r)$ denotes the value of the density function of Gaussian variable $\xi$ at point $r$. This confirms the fact that the centered variable 
$$
\sqrt{A} \left ( \frac{N'(\infty)}{A} - x^* \right ), 
$$
like $\xi_A$, also converges in distribution towards Gaussian variable $\xi$.
\label{Rs}
\end{remark}

%%%%%%%%%%%%%%%%%%%%%%%%%%%%%%%%%%%%%%%%
\section{Proof of Lemma~\ref{L0}}
\label{PL0}
%%%%%%%%%%%%%%%%%%%%%%%%%%%%%%%%%%%%%%%%
The proof of Lemma~\ref{L0} proceeds in three steps.

\textbf{(a)} We first prove the analytic continuation of $\mathrm{F}_A$ to the product 
$\mathbb{D} \times \mathbb{C}$. As in Section~\ref{sec:model-main-results}, let $M'(t)$ denote the number of customers in the $M/M/\infty$ queue with Poisson arrival process of rate $\beta$ and i.i.d.\ ``service times'' with exponential distribution of parameter $\theta$. Departures for the occupancy process $M$ for \textit{M}-customers stem from both service completion and departures due to impatience, while the departures for process $M'$ come from impatience only. The random processes $(N,M)$ and $M'$ can be coupled in such a way that $N$ and $M'$ are independent, $M'(0) = M(0)$ and
\[ M(t) \leqslant M'(t), \quad t \geqslant 0. \]
Letting $t \to \infty$, this readily entails that
\[ \mathbb{P}(N(\infty) = n, M(\infty) = m) \leqslant \mathbb{P}(M'(\infty) \geqslant m), \qquad 
(n, m) \in \mathbb{N}^2, \]
where $M'(\infty)$ has a Poisson distribution with parameter $A = \beta/\theta$. The latter inequality ensures that $\mathbb{P}(N(\infty) = n, M(\infty) = m) = O(A^m/m!)$ for large $m$ and the power series defining 
$\mathrm{F}_A(\UU,\VV)$ is thus convergent for all $\UU \in \mathbb{D}$ and $\VV \in \mathbb{C}$. Function $\mathrm{F}_A$ is therefore analytically defined in $\mathbb{D} \times \mathbb{C}$.

\textbf{(b)} We now consider the extension of $\mathrm{F}_A$ to the product 
$\mathbb{D}(0,1/\varrho) \times \mathbb{D}$. Summing all equations~\eqref{Kol0} with respect to index $m \geqslant 0$, we have
$$
\alpha \, \mathbf{Q}_A(n) + \mu \, n \sum_{m \geqslant 0} 
\frac{\pmb{\Pi}_A(n,m)}{n+m} = \alpha \, \mathbf{Q}_A(n-1) + 
\mu \, (n+1) \sum_{m \geqslant 0} 
\frac{\pmb{\Pi}_A(n+1,m)}{n+m+1}
$$
where we set $\mathbf{Q}_A(n) = \sum_{m \geqslant 0} \pmb{\Pi}_A(n,m)$, hence 
\begin{equation}
\alpha \, \mathbf{Q}_A(n) = 
\mu (n+1) \sum_{m \geqslant 0} \frac{\pmb{\Pi}_A(n+1,m)}{n+m+1}, 
\qquad n \in \mathbb{N}.
\label{MargNA}
\end{equation}
Let
\[ \pmb{\varepsilon}_A(n+1) = \frac{1}{\mathbf{Q}_A(n+1)} \, 
\left | (n+1) \sum_{m \geqslant 0} \frac{\pmb{\Pi}_A(n+1,m)}{n+m+1} - 
\mathbf{Q}_A(n+1) \right |. \]
Considering the right-hand side of~\eqref{MargNA} for large $n$, we calculate 
\begin{align}
\pmb{\varepsilon}_A(n+1) & = 
\frac{1}{\mathbf{Q}_A(n+1)} \, 
\sum_{m \geqslant 0} \frac{m}{n+m+1} \cdot \pmb{\Pi}_A(n+1,m)\\
& \leqslant \frac{1}{n+1} \cdot 
\frac{\mathbb{E}(M(\infty) ; N(\infty) = n+1)}{\mathbf{Q}_A(n+1)}
\nonumber
\end{align}
hence
\begin{equation}
\pmb{\varepsilon}_A(n+1) \leqslant \frac{1}{n+1} \cdot 
\mathbb{E}(M(\infty) \, \vert N(\infty) = n+1).
\label{MargNAbis}
\end{equation}
Since $M \leqslant M'$ and by the independence of random variables $M'$ and $N$, we have
$\mathbb{E}(M(\infty) \, \vert N(\infty) = n+1) \leqslant \mathbb{E}(M'(\infty) ) = A$. It thus follows from upper bound~\eqref{MargNAbis} that $\pmb{\varepsilon}_A(n+1) \rightarrow 0$ when $n \uparrow +\infty$ hence, after equality~\eqref{MargNA}, $\alpha \, \mathbf{Q}_A(n) \sim \mu \mathbf{Q}_A(n+1)$, that is,
$$
\frac{\mathbf{Q}_A(n+1)}{\mathbf{Q}_A(n)} \sim \varrho
$$
for large $n$. We conclude that the power series $\mathrm{F}_A(\UU,\VV)$ converges for all 
$\UU \in \mathbb{D}(0,\frac{1}{\varrho})$ and $\VV \in \mathbb{D}$. 

\textbf{(c)} Let $\overline{S}$ be the set of points $(\UU,\VV) \in \mathbb{C}^2$ where the power series 
$\mathrm{F}_A(\UU,\VV)$ converges absolutely and $S$ the interior of $\overline{S}$. We let
$$
S_0 = \{(\UU,\VV) \in S, \; \UU \VV \neq 0\}
$$
and consider the mapping $\lambda:(\UU,\VV) \in S_0 \mapsto (\log\vert \UU \vert, \log \vert \VV \vert) \in \mathbb{R}^2$. By~\cite[Chap.I, Th\'eor\`eme 3]{CHA90}, it is known that the set $S$ is logarithmically convex, that is, the image $\lambda(S_0)$ is convex in $\mathbb{R}^2$. 

Now, by Item \textbf{(a)} above, $S$ contains $\mathbb{D} \times \mathbb{C}$ hence the image 
$\lambda(S_0)$ contains the square $(-\infty,0) \times (-\infty,+\infty)$ in $\mathbb{R}^2$. Similarly, by Item \textbf{(b)} above, $S$ contains $\mathbb{D}(0,1/\varrho) \times \mathbb{D}$, hence the image 
$\lambda(S_0)$ contains the square $(-\infty, - \log \varrho) \times (-\infty,0)$. By the convexity property of $\lambda(S_0)$, we then deduce that $\lambda(S_0)$ contains its convex envelope and thus also the complementary square $(0,-\log \varrho) \times (0,+\infty)$, so that we eventually have 
$$
\lambda(S_0) \supset (-\infty,-\log\varrho) \times (-\infty,+\infty).
$$
The convergence domain $S$ of power series $\mathrm{F}_A(\UU,\VV)$ therefore contains the product 
$\pmb{\Omega} = \mathbb{D}(0,1/\varrho) \times \mathbb{C}$. Function $\mathrm{F}_A$ is thus analytically defined in $\pmb{\Omega}$, as claimed.

%%%%%%%%%%%%%%%%%%%%%%%%%%%%%%%%%%%%%%%%
\section{Proof of Proposition~\ref{P3}}
\label{PL1}
%%%%%%%%%%%%%%%%%%%%%%%%%%%%%%%%%%%%%%%%
In the following, we further assume that $\UU \notin \, (-\infty,0]$, $v \notin \; (-\infty,0]$ and $\log$ denotes the principal determination of the logarithm over the cut plane $\mathbb{C} \setminus \; (-\infty,0]$. By definition of $\mathrm{F}_A$, write
$$
\mathrm{F}_A(\UU,\VV) = I + J + K + L
$$
where
$$
I = \sum_{n \geqslant 1,m\geqslant 1} \mathbb{P}(N(\infty) = n, M(\infty) = m) \, u^n v^m,
$$
$$
J = \sum_{m \geqslant 1} \mathbb{P}(N(\infty) = 0, M(\infty) = m) \, v^m, \; \; 
K = \sum_{n \geqslant 1} \mathbb{P}(N(\infty) = n, M(\infty) = 0) \, u^n
$$
and $L = \pmb{\Pi}_A(0,0)$.

\textbf{(a)} First consider the sum $I$. With the change scale $n = Ax$ and $m = Ay$, 
$x > 0$, $y > 0$ and by Theorem~\ref{T1}, we have
\[ \mathbb{P}(N(\infty) = Ax, M(\infty) = Ay) u^{Ax}v^{Ay} \asymp \exp(-A \, h_{\UU,\VV}(x,y)) \]
for large $A$ ($f \asymp g$ meaning that $-\log f/A \sim -\log g/A$ when 
$A \uparrow +\infty$) where
\begin{align}
h_{\UU,\VV}(x,y) = & \, \Phi(x) - x \, \log \UU + \Psi(y) - y \, \log \VV
\label{h1} \\
= & \, \Phi(x) - x \, \log r + \Psi(y) - y \, \log s - i(x \zeta + y \eta)
\nonumber
\end{align}
with $\UU = r \, e^{i\zeta}$ and $v = s \, e^{i\eta}$ as in the statement of Proposition~\ref{P3}. For given real 
$r = \vert \UU \vert > 0$, $s = \vert \VV \vert > 0$, and after the respective definitions~\eqref{eq:Phi} and~\eqref{eq:Psi} of functions $\Phi$ and $\Psi$, the real-valued function $h_{r,s}: (x,y) \in \mathbb{R}^{+*2} \mapsto 
h_{r,s}(x,y)$ is easily shown to have a unique minimum at point $(X_\UU,Y_\VV) \in \mathbb{R}^{+*2}$ given by
\begin{equation}
X_\UU = \frac{\varrho \, r}{1 - \varrho r}, \quad Y_\VV = s.
\label{h2}
\end{equation}
The corresponding value $h_{r,s}(X_\UU,Y_\VV) = \Phi(X_\UU) - X_\UU \log r + \Psi(Y_\VV) - Y_\VV \log s$ 
is easily calculated as $h_{r,s}(X_\UU,Y_\VV) = -\log(1-\varrho) + \log(1 - \varrho r) - (s-1)$ so that
\begin{equation}
\exp \left( -A \, h_{r,s}(X_\UU,Y_\VV) \right) = \left ( \frac{1-\varrho}{1-\varrho r} \right )^A 
e^{A(s-1)}
\label{h3}
\end{equation}
while the second-order derivatives of $h_{r,s}$ at $(X_\UU,Y_\VV)$ are given by
$$
\frac{\partial^2 h_{r,s}}{\partial x^2}(X_\UU,Y_\VV) = 
\frac{(1-\varrho r)^2}{\varrho r}:= a_r, \qquad 
\frac{\partial^2 h_{r,s}}{\partial y^2}(X_\UU,Y_\VV) = 
\frac{1}{s} := b_s
$$
and $\partial_{xy}^2 h_{r,s}(X_\UU,Y_\VV) = 0$. Estimating the discrete sum I over the lattice 
$\mathbb{N}^{*2}$ by a Riemann integral over $\mathbb{R}^{+*2}$ with integration step $1/A$ and using asymptotics~\eqref{AsymptPA} for $\pmb{\Pi}_A(Ax,Ay)$, we further have
\begin{align}
I \sim & \, 
\int_{0^+}^{+\infty} A \, \mathrm{d}x \int_{0^+}^{+\infty} A \, \mathrm{d}y 
\; \pmb{\Pi}_A(Ax,Ay) \cdot u^{Ax} v^{Ay} 
\label{h4} \\
\sim & \, \frac{A}{2\pi} \int_{0^+}^{+\infty} \int_{0^+}^{+\infty} g(x,y) \, 
e^{-A \, h_{\UU,\VV}(x,y)} \mathrm{d}x\mathrm{d}y
\nonumber
\end{align}
with function $h_{\UU,\VV}$ introduced in~\eqref{h1}. Now applying the asymptotics~\eqref{Lapla0BIS} to evaluate the (complex-valued) integral~\eqref{h4} with help of~\eqref{h3}, we then obtain
\begin{multline*}
I \sim \,\frac{A}{2\pi} \times g(X_\UU,Y_\VV) \cdot 
e^{-A \, h_{r,s}(X_\UU,Y_\VV) + iA (\zeta X_\UU + \eta Y_\VV)} \; \times \\
\, \exp \left [ - \frac{A \zeta^2}{2 \, a_r} \right ] 
\exp \left [ - \frac{A \eta^2}{2 \, b_s} \right ] \, \sqrt{\frac{2\pi}{A \, a_r}} 
\sqrt{\frac{2\pi}{A \, b_s}}.
\end{multline*}
Using successively the expressions~\eqref{h3} for $e^{-A \, h_{r,s}(X_\UU,Y_\VV)}$ and the second-order derivatives $a_r$ and $b_s$ of $h_{r,s}$ at point $(X_\UU,Y_\VV)$, together with the definition of $g$ in~\eqref{AsymptPA} eventually reduces the latter estimate of $I$ to
\begin{multline}
I \sim \, \left ( \frac{1-\varrho}{1 - \varrho r}\right )^A e^{A(s-1)} 
\exp \left [ iA \left ( \frac{\varrho r \zeta }{1-\varrho r} + s \eta \right ) 
\right ] \; \times
\\
\, \exp \left [ - \frac{A}{2} \left ( \frac{\varrho r \zeta^2}{(1 - \varrho r)^2} + 
s \eta^2 \right ) \right ] \, G_0(\UU,\VV)
\label{h5}
\end{multline}
with coefficient
$$
G_0(\UU,\VV) = \frac{g(X_\UU,Y_\VV)}
{\partial^2_{xx}h_{r,s}(X_\UU,Y_\VV)\partial^2_{yy}h_{r,s}(X_\UU,Y_\VV)} = 
\left ( \frac{1-\varrho}{1-\varrho r}\right ) 
\bigl [ s + \varrho r (1-s) \bigr ]^{\frac{\alpha}{\theta}(1-r)-c}.
$$
To specify the definition domain of function $G_0$, first assume $s \leqslant 1$; then the argument 
$s + \varrho r (1-s) \geqslant s > 0$ for all $r$ (and thus also for $r < 1/\varrho$); now if $s > 1$, 
$s + \varrho r (1-s)$ is positive if and only if $\varrho r < 1 + 1/(s-1)$, which is fulfilled if $\varrho r < 1$. We thus conclude that function $G_0$ is well-defined and continuous over $\pmb{\Omega}$, and thus also in the subset $\pmb{\Omega}'$. At point $(\UU,\VV) = (1,1)$, in particular, we clearly have $r = s = 1$ so that $G_0(1,1) = 1$.

\textbf{(b)} Now address the second term $J$. By Theorem~\ref{T1} again, we can write
$$
\mathbb{P}(N(\infty) = 0, M(\infty) = Ay) \asymp e^{-A \, H(0,y)} \VV^{Ay} = 
e^{-A \, h_{1,\VV}(0,y)}
$$
with the notation~\eqref{h1} for function $h_{1,\VV}$. For $s = \vert \VV \vert > 0$, the real-valued function 
$y \in \mathbb{R}^{+*} \mapsto h_{1,s}(0,y)$ has a unique minimum at $y = Y_s = s > 0$, with value
$$
h_{1,s}(0,Y_\VV) = -\log(1-\varrho) - (s-1).
$$
The module of the sum $J$ is therefore of order $\vert J \vert \asymp e^{-A \, h_{1,s}(0,Y_\VV)} = 
(1-\varrho)^A e^{A(s-1)}$ and, after the estimate~\eqref{h5} of $I$, the ratio $I/J$ is of order
$$
\frac{\vert I \vert}{\vert J \vert} \asymp \frac{1}{(1-\varrho r)^A}
$$
and thus tends to 0 when $A \uparrow +\infty$ and for $(\UU,\VV) \in \pmb{\Omega}'$. We conclude that $I$ dominates $J$ for large $A$. 

\textbf{(c)} As to the third term $K$, it is similarly verified that $\vert I \vert 
/ \vert K \vert \asymp e^{As}$ with $s = \vert \VV \vert > 0$. As the latter ratio tends to $+\infty$ when $A \uparrow +\infty$, $I$ also dominates $K$ for large $A$. Finally, 
$$
\frac{\vert I \vert }{ \vert L \vert} \asymp \frac{e^{As}}{(1-\varrho r)^A}
$$
and $I$ also dominates $L$ for large $A$. 

After~\eqref{h5} and the latter discussion, asymptotics~\eqref{HEDP2} for $\mathrm{F}_A(\UU,\VV)$, 
$(\UU,\VV) \in \pmb{\Omega}'$ eventually follows.

% \section*{Acknowledgments}
% We would like to acknowledge the assistance of volunteers in putting
% together this example manuscript and supplement.

% %%%%%%%%%%%%%%%%%%%%%%%%%%%%%%%%%%%%

% %%%%%%%%%%%%%%%%%%%%%%%%%%%%%%%%%%%%%

% \bibliographystyle{siamplain}
% \bibliography{references}
\end{document}